\documentclass{article}

\usepackage[dvipsnames, svgnames, x11names]{xcolor}

\usepackage{amssymb}
\usepackage{amsmath,amsthm}
\usepackage[a4paper]{geometry}  
\usepackage{enumerate}
 \usepackage{graphicx}
 
 \usepackage{subcaption}

\usepackage[colorlinks=true, allcolors=blue]{hyperref}

\newcommand{\keywords}[1]{\textbf{Keywords:} #1}

\newtheorem{theorem}{Theorem}[section]

\newtheorem{proposition}{Proposition}[section]
\newtheorem{lemma}{Lemma}[section]
\newtheorem{remark}{Remark}

\newtheorem{corollary}{Corollary}[section]
\newtheorem{example}{Example}[section]

\usepackage{authblk}

\usepackage{algorithm}
\usepackage{algorithmic}

\usepackage{setspace}

\usepackage{natbib}

\bibpunct{[}{]}{,}{n}{}{;}

\title{Periodicity in Hedge-myopic system and an asymmetric NE-solving paradigm for two-player zero-sum games}
\author[a,b]{Xinxiang Guo}

\author[b]{Yifen Mu\thanks{Corresponding author: mu@amss.ac.cn}}

\author[c]{Xiaoguang Yang}

\affil[a]{School of Mathematical Sciences, University of Chinese Academy of Sciences}

\affil[b]{The Key Laboratory of Systems and Control, Institute of Systems Science, Academy of Mathematics and Systems Science, Chinese Academy of Sciences}

\affil[c]{Key Laboratory of Management, Decision and Information System, Institute of Systems Science, Academy of Mathematics and Systems Science,
Chinese Academy of Sciences}

\date{}

\begin{document}
\maketitle

\begin{abstract}
     In this paper, we consider the $n \times n$ two-payer zero-sum repeated game in which one player (player X) employs the popular Hedge (also called multiplicative weights update) learning algorithm while the other player (player Y) adopts the myopic best response. We investigate the dynamics of such Hedge-myopic system by defining a metric $Q(\textbf{x}_t)$, which measures the distance between the stage strategy $\textbf{x}_t$ and Nash Equilibrium (NE) strategy of player X. We analyze the trend of $Q(\textbf{x}_t)$ and prove that it is bounded and can only take finite values on the evolutionary path when the payoff matrix is rational and the game has an interior NE.
     Based on this, we prove that the stage strategy sequence of both players are periodic after finite stages and the time-averaged strategy of player Y within one period is an exact NE strategy. 
     Accordingly, we propose an asymmetric paradigm for solving two-player zero-sum games. 
     For the special game with rational payoff matrix and an interior NE, the paradigm can output the precise NE strategy; for any general games we prove that the time-averaged strategy can converge to an approximate NE. In comparison to the NE-solving method via Hedge self-play, this HBR paradigm exhibits faster computation/convergence,better stability and can attain precise NE convergence in most real cases. 
\end{abstract}


\section{Introduction}

Game theory is widely used to model interactions and competitions among self-interested and rational agents in the real world \cite{fudenberg1991game, narahari2014game, osborne1994course}. In such situations, the utility of each agent is determined by the actions of all the other agents, leading to the solution concept of equilibrium, in which Nash equilibrium (NE) is a central one. This makes the NE-solving one of the most significant problems in game theory, which however is very hard for general games. 

The notion of NE \cite{nash1950equilibrium, nas1951non} aims to describe a stable state where each participant makes the optimal choice considering the strategies of others and thus has no incentive to change the strategy unilaterally. Due to the mutual influence of participants in the game and the existence of possible multiple equilibrium points in different scenarios, NE-solving is a difficult problem, which has been proven to be PPAD-hard \cite{chen2006settling, daskalakis2009complexity, papadimitriou1994complexity}. For the special two-player zero-sum games, linear programming \cite{karmarkar1984new, von2007theory} provides a powerful method to solve NE in polynomial time \cite{van2020deterministic}. However, in practical scenarios, because of large scalability issues, imperfect information and the complexity of multiple-stage dynamics, two-player zero-sum game is still a subject of ongoing investigation and attracts attentions from researchers in different fields \cite{mertikopoulos2018cycles, daskalakis2011near, perolat2022mastering, perolat2021poincare}. Especially, participants are often not perfectly rational, leading to research which aimed at approximating NE from a learning perspective \cite{fudenberg1998theory}.

Concerning learning in games, there has been a long history and plentiful literature, where a lot of learning algorithms have been proposed according to different settings. Under the Fictitious Play algorithm \cite{fudenberg1998theory}, the empirical distribution of actions taken by each player converges to NE if the stage game has generic payoffs and is $2\times 2$ \cite{robinson1951iterative} or zero-sum game \cite{miyasawa1961convergence} or potential game \cite{Monderer1996-sg}. When each player employs the no-regret algorithm \cite{cesa1997use, cesa2006prediction} to determine their stage strategy in repeated games, their time-averaged strategy profile converges to the coarse correlated equilibrium in general-sum games \cite{hart2000simple} and to NE in two-player zero-sum games. In imperfect-information extensive-form games, Zinkevich et al. \cite{zinkevich2007regret} proposed the counterfactual regret minimization (CFR) algorithm and proved its convergence to NE in the two-player zero-sum setting. Based on these methods, lots of variants were proposed and widely used in solving equilibrium in complicated games \cite{brown2018superhuman, brown2019solving, brown2019deep, lanctot2009monte, moravvcik2017deepstack, sayin2022fictitious}. Note that in all these works, every player in the game adopts the same learning algorithm and the convergence results are based on the term of time-averaged strategy. 

However, further investigation on the learning dynamics shows that even in simple game models, basic learning algorithms can lead to highly complex behavior and may not converge \cite{papadimitriou2016nash, piliouras2014optimization, sato2002chaos, piliouras2014persistent, hofbauer1996evolutionary}. Palaiopanos et al. \cite{palaiopanos2017multiplicative} discovered specific instances of $2 \times 2$ potential games where the behavior of multiplicative weights update (MWU) algorithm exhibits bifurcation at the critical value of its step size. Bailey and Piliouras \cite{bailey2018multiplicative} showed that in two-player zero-sum games, when both players adopt the MWU algorithm, the system dynamics deviate from equilibrium and converge towards boundary. 
Mertikopoulos et al. \cite{mertikopoulos2018cycles} studied the regularized learning algorithms in two-player zero-sum games and proved the Poincar\'e recurrence of the system behavior, implying the impossibility of convergence to NE from any initial strategy profile. Perolat et al. \cite{perolat2021poincare} extended the results of Poincar\'e recurrence from normal-form games to two-player zero-sum imperfect-information games and built an algorithm to approximate NE by solving a series of regularized game with unique NE. Generally speaking, there is no systematic framework for analyzing the limiting behavior of these repeated games \cite{jordan1993three, shapley1964some}. 

For the asymmetric case, as emphasized in \cite{candogan2013dynamics}, the limiting behavior of dynamic processes where players adhere to different update rules is an open question, even for potential games. Therefore, related theoretical analysis of such system is extremely rare. Our previous work \cite{guo2023optimal, guo2023taking} studied such dynamical system where one player employs the Hedge algorithm and the other player takes the globally or locally optimal strategy in finitely repeated two-player zero-sum games and proved its periodicity when the game is $2\times 2$. As a byproduct, our investigation reveals that the detailed understanding about the dynamics can facilitate the design of novel algorithms with special properties, thereby suggesting a promising avenue for advancing learning algorithms. 

This paper will consider the general $n\times n$ zero-sum stage game and investigate the dynamics of repeated game under asymmetric updating rules. To be specific, we will study the dynamics of the repeated game where one player (player X) employs the Hedge algorithm to update his stage strategy and the other player (player Y) adopts the according myopic best response to the stage strategy of player X. The main contributions of this paper can be summarized as follows.

\begin{enumerate}[(1)]
    \item This paper considers the Hedge-myopic system and investigates its dynamic by analyzing the trend of a quantity called $Q(\textbf{x})$ based on the Kullback-Leibler divergence, which measures the distance between the stage strategy $\textbf{x}$ and the NE strategy of player X.  For the game with rational payoff matrix and an interior NE, we prove that along the strategy sequence $\textbf{x}_t$ the Q-sequence $Q(\textbf{x}_t)$ is bounded and $\textbf{x}_t$ can only take finite values on the evolutionary path. This implies that the strategy sequence $\textbf{x}_t$ of player X will not converge to the NE strategy and justifies the finding in the literature.
    \item  Using the dynamic property, this paper theoretically proves that the stage strategy sequences of both players are periodic after finite stages for the game with rational payoff matrix and an interior NE. 
    Additionally, the time-averaged strategy of player Y within one period is an exact NE strategy.    
    \item Based on the theoretical results, this paper proposes an asymmetric paradigm called HBR for solving NE in two-player zero-sum games. For the special game with rational payoff matrix and an interior NE, the paradigm can output the precise NE strategy; for any general games we prove that the time-averaged strategy can converge to an approximate NE. In comparison to the NE-solving method via Hedge self-play, this HBR paradigm exhibits faster computation/convergence, better stability and can attain precise NE convergence in most real cases. 
\end{enumerate}

\textbf{Paper Organization:} Section \ref{sec_problem_formulation} provides the preliminary knowledge and problem formulation; Section \ref{sec_main_result} presents the main results regarding the periodicity of the system behavior; Section \ref{sec_HBR} proposes an asymmetric paradigm for solving NE and gives the experiment results; Section \ref{sec_conclusion} concludes the paper.

\section{Preliminary and Problem Formulation}\label{sec_problem_formulation}

\subsection{Online Learning and Hedge Algorithm}

Hedge algorithm is a popular no-regret learning algorithm proposed by Freund and Schapire, based on the context of boost learning \cite{freund1997decision}. Hedge algorithm is also known as weighted majority \cite{littlestone1994weighted} or exponential weighted average prediction \cite{cesa2006prediction}, or multiplicative weights update \cite{arora2012multiplicative}. 

Consider the online learning framework known as learning with expert advice \cite{cesa2006prediction}. In this framework, the decision maker is a forecaster whose goal is to predict an unknown sequence $q_1,q_2,\cdots$, where $q_t$ belongs to an outcome space $\mathcal{Q}$. The prediction of the forecaster at time $t$, denoted by $\hat{p}_t$, is assumed to belong to a convex subset $\mathcal{D}$ of $\mathcal{Q}$. At each time $t$, the forecaster receives a finite set of expert advice $\{f_{i,t}\} \in \mathcal{D}: i=1, 2, \cdots, N\}$, then the forecaster computes his own guess $\hat{p}_t$ based on $\{f_{i,t}\}$. Subsequently the true outcome $q_t$ is revealed. Predictions of the forecaster and experts are scored using a non-negative loss function $\ell:\mathcal{D}\times\mathcal{Q}\rightarrow\mathbb{R}$ and the \textit{cumulative regret} is introduced to measure how much better the forecaster could have done compared to how he did in hindsight, which is defined to be
\begin{equation*}
    R_{t} = \max\limits_{i = 1, 2, \cdots, N} \left\{ \sum\limits_{\tau=1}^t \left(\ell(\hat{p}_\tau,q_\tau)-\ell(f_{i,\tau},q_\tau)\right)\right\}.
\end{equation*}   

By the Hedge algorithm, the prediction $\hat{p}_t$ at time $t$ is taken as the weighted average of the predictions from the experts, i.e.,  $\hat{p}_t = \sum\nolimits_{i=1}^N w_{i, t-1} f_{i, t}$, where
\begin{equation*}
    w_{i, t-1} = \frac{e^{-\eta \sum\nolimits_{\tau=1}^{t-1} \ell(f_{i,\tau}, q_{\tau})}}{\sum\nolimits_{j=1}^N e^{-\eta \sum\nolimits_{\tau=1}^{t-1} \ell(f_{j,\tau}, q_{\tau})}}, \quad i=1, 2, \cdots, N,\ t\geq 1.
\end{equation*}
Then, the prediction sequence has the following regret bound.
\begin{theorem}[{Theorem~2.2 of \cite{cesa2006prediction}}]\label{thm_regret_bound}
    Assume that the loss function $\ell$ is convex in its first argument and takes values in $[0, 1]$. For any $t$ and $\eta > 0$, and for all $q_1, q_2, \cdots ,q_t \in \mathcal{Q}$, the regret for the Hedge algorithm satisfies
    \begin{equation*}
        R_t \leq \frac{\ln N}{\eta} + \frac{t \eta}{8}.
    \end{equation*}
    In particular, for $\eta = \sqrt{ \frac{8\ln N }{t}}$, the upper bound becomes $\sqrt{\frac{t}{2}\ln N}$.
\end{theorem}
Theorem \ref{thm_regret_bound} implies that the time-averaged regret of the Hedge algorithm goes to zero as $t$ increases, i.e., $\lim \nolimits_{t\to \infty} R_t /t = 0$, which is the so-called \textit{no-regret} property. This theorem will be used to prove the convergence of the time-averaged strategy for the Hedge-myopic system in Section \ref{sec_main_result}. 

\subsection{Normal-form zero-sum Game}

Consider a two-player zero-sum normal-form game $\Gamma$. The players are called player X and player Y. Suppose there are $n, n\geq 2$ feasible actions for each player. We denote the action set of player X by $\mathcal{I} = \{1, 2, \cdots, n\}$ and the action set of player Y by $\mathcal{J}= \{1, 2, \cdots, n\}$. For each action profile $(i, j)$, the payoff obtained by player Y is $a_{ij}$ and thus the payoff obtained by player X is $-a_{ij}$ since the game is zero-sum. Naturally, the payoff of the game is shown by a matrix $A = \{a_{ij}\}_{i\in \mathcal{I}, j\in \mathcal{J}}$. The matrix is called the payoff matrix for player Y and the loss matrix for player X. 
A mixed strategy of a player is a probability distribution over his action set. Denote the mixed strategy of player X and player Y by $\textbf{x}\in \Delta(\mathcal{I})$ and $\textbf{y}\in \Delta(\mathcal{J})$ respectively. The bold font is used to emphasize that $\textbf{x}$ and $\textbf{y}$ are both vectors.
Given the mixed strategy profile $(\textbf{x}, \textbf{y})$, the payoff of player Y is $\textbf{x}^TA\textbf{y}$ and the payoff of player X is $-\textbf{x}^TA\textbf{y}$. 

Write the NE strategy profile of the game as $(\textbf{x}^\ast, \textbf{y}^\ast)$. Then the value of the game is 
\begin{equation}\label{Def_value}
    v^\ast := (\textbf{x}^\ast)^TA\textbf{y}^\ast.
\end{equation}
Denote the support of $\textbf{x}^\ast$ by $\operatorname{supp}(\textbf{x}^\ast) = \{i\in \mathcal{I}: x_i^\ast > 0\}$ and the support of $\textbf{y}^\ast$ by $\operatorname{supp}(\textbf{y}^\ast) = \{j\in \mathcal{J}: y_j^\ast > 0\}$. A NE is said to be interior if $\operatorname{supp}(\textbf{x}^\ast) = \mathcal{I}$ and $\operatorname{supp}(\textbf{y}^\ast) = \mathcal{J}$. 

A strategy profile $(\textbf{x}, \textbf{y})$ is called a $\varepsilon-$Nash equilibrium ($\varepsilon-$NE) if for all $\textbf{x}^\prime\in \Delta(\mathcal{I})$ and $\textbf{y}^\prime\in \Delta(\mathcal{J})$, we have
\begin{equation*}
    \textbf{x}^T A \textbf{y} \geq \textbf{x}^T A \textbf{y}^\prime - \varepsilon \quad\text{and}\quad  \textbf{x}^T A \textbf{y} \leq (\textbf{x}^\prime)^T A \textbf{y} + \varepsilon.
\end{equation*}
Given a strategy profile $(\textbf{x}, \textbf{y})$, the \textit{exploitability} \cite{lockhart2019computing} of strategy $\textbf{x}$ is defined as the potential gain for player Y if she switches to the best response of strategy $\textbf{x}$ suppose that $\textbf{x}$ is fixed, i.e., 
\begin{equation*}
    e_x(\textbf{x}, \textbf{y}) \triangleq \max\nolimits_{\textbf{y}^\prime\in \Delta(\mathcal{J})} \textbf{x}^T A \textbf{y}^\prime - \textbf{x}^T A \textbf{y}.
\end{equation*}
Similarly, the exploitability of the strategy $\textbf{y}$ is defined to be 
\begin{equation*}
    e_y(\textbf{x}, \textbf{y}) \triangleq  \textbf{x}^T A \textbf{y} - \min\nolimits_{\textbf{x}^\prime\in \Delta(\mathcal{I})} (\textbf{x}^\prime)^T A \textbf{y}.
\end{equation*}
Intuitively, if a strategy has low exploitability, it is difficult for the opponent to take advantage of it, while a strategy with high exploitability can be effectively exploited by the opponent. Given a strategy profile $(\textbf{x}, \textbf{y})$, a common metric to measure its distance to NE, called \textit{Nash Distance} ($\operatorname{ND}$) \cite{lockhart2019computing}, is 
\begin{equation*}
    \operatorname{ND}(\textbf{x}, \textbf{y}) \triangleq e_x(\textbf{x}, \textbf{y}) + e_y(\textbf{x}, \textbf{y}).
\end{equation*}
Since the game is zero-sum, $\operatorname{ND}(\textbf{x}, \textbf{y}) = \max\nolimits_{\textbf{y}^\prime\in \Delta(\mathcal{J})} \textbf{x}^T A \textbf{y}^\prime - \min\nolimits_{\textbf{x}^\prime\in \Delta(\mathcal{I})} (\textbf{x}^\prime)^T A \textbf{y}$.

If the game admits an interior NE strategy $\textbf{x}^\ast$ of player X, we denote the cross entropy between strategy $\textbf{x}$ and $\textbf{x}^\ast$ by function $Q(\mathbf{x})$, i.e.,
\begin{equation}\label{Def_Q(x)}
    Q(\mathbf{x}) \triangleq - \sum\limits_{i=1}^n x_i^\ast \ln x_{i}
\end{equation}
where $0 < x_i < 1$ for all $i\in \mathcal{I}$. It is easy to see that $Q(\textbf{x}) > 0$ for all $\textbf{x}$.

In the following, we will focus on infinitely repeated game, that is to say, let the game $\Gamma$ be repeated for infinite times. We denote $\textbf{x}_t$, $\textbf{y}_t$ the stage strategy of player X and player Y at time $t$ respectively. Then, the instantaneous expected payoff of player Y is $\textbf{x}_t^T A\textbf{y}_t$. Different stage strategy updating rules would lead to different stage strategy sequences of player X and Y, which form different game dynamic systems. Below we will study the dynamic characteristics of the repeated games driven by the Hedge algorithm and the myopic best response. 

\subsection{Problem Formulation}

First, let player X update his stage strategy according to the Hedge algorithm. Specifically, the strategy of player X at time $t$ is updated by
\begin{equation}\label{Eq_hedge_formula}
    x_{i, t} = \frac{\exp(-\eta \sum\limits_{\tau=1}^{t-1}e_i^TA\textbf{y}_{\tau})}{\sum\limits_{j=1}^n\exp(-\eta \sum\limits_{\tau=1}^{t-1}e_j^TA\textbf{y}_{\tau})},\quad i = 1, 2, \cdots, n,
\end{equation}
where $\eta$ is called \textit{the learning rate}, which is a constant parameter. In this paper, we assume that  $\eta$ is sufficiently small and determined by the payoff matrix $A$.

From formula \eqref{Eq_hedge_formula}, $\textbf{x}_t$ is fully determined by $\textbf{y}_1$, $\textbf{y}_2$, ...,  $\textbf{y}_{t-1}$. Further, we can compute $\textbf{x}_{t}$ from $\textbf{x}_{t-1}$ and $\textbf{y}_{t-1}$ and get  
\begin{equation}\label{Eq_xt+1_and_xt_yt}
    x_{i, t} = \frac{x_{i, t-1} \exp(-\eta e_i^T A \textbf{y}_{t-1})}{\sum\nolimits_{j=1}^n x_{j, t-1} \exp(-\eta e_j^T A \textbf{y}_{t-1})}, \quad \forall\ i= 1, 2, \cdots, n.
\end{equation}

Then, let player Y only consider maximizing her instantaneous expected payoff and take myopic best response to $\textbf{x}_t$ at each time $t$. 
In most cases, the myopic best response is unique and is pure strategy. When the best response is not unique, we stipulate that player Y chooses the pure strategy with the smallest subscript in $\mathcal{J}$ , i.e., 
\begin{equation}\label{Eq_yt_formula}
    \textbf{y}_t = BR(\textbf{x}_t)  \triangleq \textbf{y}^{\bar{j}},\quad \text{where}\  \bar{j} = \min\{j\in \mathcal{J}: \textbf{x}_t^TA \textbf{y}^j = \max\limits_{\textbf{y}\in \Delta(\mathcal{J})} \textbf{x}_t^T A \textbf{\textbf{y}}\}
\end{equation}
and $\textbf{y}^j$ is a pure-strategy vector with only the $j$-th element is 1 and all the other elements are 0. 
Then, the strategy of player Y at each stage is well-defined and is pure strategy. Apparently $\textbf{y}_t$ is totally determined by $\textbf{x}_t$. Combined with \eqref{Eq_xt+1_and_xt_yt}, we know that $\textbf{x}_{t+1}$ is fully determined by $\textbf{x}_t$, with no randomness involved.

Given the action rule of player X and Y as above, the infinitely repeated game is intrinsically determined and the system is called \textit{the Hedge-myopic system}.
In this paper, we will study the dynamic characteristics of such a dynamical system and try to answer the questions like: Is the system periodic? Does the system converge?

\section{Main Results}\label{sec_main_result}

\subsection{Rational Games with an Interior Equilibrium}

In this session, we prove that for the game with rational payoff matrix and an interior NE, the dynamics of the Hedge-myopic system is periodic after finite stages. 

We state some assumptions below: 

\textbf{Assumption 1}:  The payoff matrix is rational and the game has an unique interior NE, denoted by $(\mathbf{x}^\ast, \mathbf{y}^\ast)$; 

\textbf{Assumption 2}: The matrices $A_i \in \mathbb{R}^{n\times n}$ are all non-singular
for $i=1,2,...,n$ where  $A_i$ is defined as 
\begin{equation}
    \begin{pmatrix}
        a_{1,1} & a_{2,1} & \cdots & a_{n,1}\\
        \vdots & \ & \  & \vdots \\
        a_{1, i-1} & a_{2, i-1} & \cdots & a_{n, i-1} \\
        a_{1, i+1} & a_{2, i+1} & \cdots & a_{n, i+1} \\
        \vdots & \ & \  & \vdots \\
        a_{1, n} & a_{2, n} & \cdots & a_{n, n} \\
        1 & 1 & \cdots & 1
    \end{pmatrix}.
\end{equation}

We note here that the uniqueness of equilibrium here is actually not a necessary condition. Through experiments, we found that as long as the game has interior equilibrium, even if it is not unique, the game system can still generate cycles. The assumption of uniqueness is a requirement in theoretical proofs. 

Recall formula \eqref{Def_Q(x)} and \eqref{Eq_hedge_formula}, we denote 
\begin{equation}\label{Def_goodness}
    Q_t = Q(\textbf{x}_t) = - \sum\limits_{i=1}^n x_i^\ast \ln x_{i,t} = \ln \left( \sum\limits_{i=1}^n e^{\eta \sum_{\tau=1}^{t-1} (v^\ast - e_i^T A \textbf{y}_\tau)} \right),
\end{equation}
where $v^\ast$ is the value of the game. Thus $Q_t$ is determined by $\textbf{y}_1$, $\textbf{y}_2$, ...,  $\textbf{y}_{t-1}$ and it can measures the ``goodness'' of $\textbf{x}_t$. Call $\{Q_t\}$ \textit{the Q-sequence}.

$Q_t$ actually is based on \textit{the Kullback-Leibler divergence} (KL divergence), which measures the level of resemblance between the strategy $\textbf{x}_t$ and the NE strategy $\textbf{x}^\ast$ and defined as  
\begin{equation}\label{Def_KL_divergence}
    \operatorname{KL}(\textbf{x}^\ast, \textbf{x}_t) = - \sum\limits_{i=1}^n x_i^\ast \ln x_{i, t} + \sum\limits_{i=1}^n x_i^\ast \ln x_{i}^\ast.
\end{equation}
Here, we omit the term $\sum\nolimits_{i=1}^n x_i^\ast \ln x_{i}^\ast$ for convenience and obviously $$\operatorname{KL}(\textbf{x}^\ast, \textbf{x}_{t+1}) > \operatorname{KL}(\textbf{x}^\ast, \textbf{x}_t) \Leftrightarrow Q_{t+1} > Q_{t}.$$ Hence, by studying the Q-sequence especially the difference between $Q_{t+1}$ and $Q_t$, we can obtain the variation in the level of resemblance between the stage strategy and the NE strategy along the time. 

By \eqref{Def_goodness}, we can calculate
\begin{align}
    Q_{t+1} - Q_t =  \ln \left(  \sum\limits_{i=1}^n x_{i, t}  e^{\eta (v^\ast - e_i^T A \textbf{y}_t)} \right)  \label{Eq_Qt+1_minus_Qt}.
\end{align}

Recall that $e^\omega \sim \omega+1$ if $\omega$ is small. Substitute this into the definition of $Q_t$,  we obtain
\begin{equation}\label{budengshi2}
    Q_{t+1}-Q_t \approx \ln \left(  \sum\limits_{i=1}^n x_{i, t} \left(1 + \eta (v^\ast - e_i^T A \textbf{y}_t)\right) \right).
\end{equation}

On the other hand, since $\textbf{y}_t$ is the myopic best response of $\textbf{x}_t$, i.e., $\textbf{x}_t^T A \textbf{y}_t = \max\nolimits_{\textbf{y}\in \Delta(\mathcal{J})} \textbf{x}_t^T A \textbf{y}$, we have  
\begin{equation}\label{budengshi1}
    \sum\limits_{i=1}^n x_{i, t}  (v^\ast - e_i^T A \textbf{y}_t) = v^\ast - \textbf{x}_t^T A \textbf{y}_t \leq 0
\end{equation}
by the Minimax theorem \cite{von2007theory}.
%which states that $v^\ast = \min\nolimits_{\textbf{x}\in %\Delta(\mathcal{I})}\max\nolimits_{\textbf{y}\in \Delta(\mathcal{J})}  \textbf{x}^T A %\textbf{y}$.

Substitute \eqref{budengshi1} to \eqref{budengshi2}, we get 
\begin{equation*}
    Q_{t+1}-Q_t \approx \ln \left(  \sum\limits_{i=1}^n x_{i, t} \left(1 + \eta (v^\ast - e_i^T A \textbf{y}_t)\right) \right) \leq 0,
\end{equation*}
indicating that roughly speaking, the value of $Q_t$ gradually decreases as time $t$ increases. 

Inspired by this rough estimation, we can prove the below theorem.

\begin{theorem}\label{Lem_Qsequence_bounded}
    There exists a positive number $M_Q$ such that $0 < Q_t \leq M_Q$ for all $t\geq 1$, i.e., the Q-sequence is bounded.
\end{theorem}

To prove Theorem 3.1, we need to study the sequence  $\textbf{x}_t$ and $Q_t$ in details. To this end, define
\begin{equation}\label{Def_Dx}
    D(\textbf{x}) \triangleq  \ln \left(  \sum\limits_{i=1}^n x_{i}  e^{\eta (v^\ast - \textbf{e}_i^T A \textbf{y}_x)} \right), \quad \textbf{x}\in \Delta(\mathcal{I})
\end{equation}
where $\textbf{y}_x =BR(\textbf{x})  \triangleq   \textbf{y}^{\bar{j}}$ with $\bar{j} = \min\{j\in \mathcal{J}: \textbf{x}^TA \textbf{y}^j = \max\nolimits_{\textbf{y}\in \Delta(\mathcal{J})} \textbf{x}^T A \textbf{y}\}$.
Take $\textbf{x} = \textbf{x}_t$, we have $D(\textbf{x}_t) = Q_{t+1} - Q_t$. 
Obviously, if $D(\textbf{x}_t) < 0$, then $Q_{t+1} < Q_t$; if $D(\textbf{x}_t) \geq 0$, then $Q_{t+1} \geq Q_t$. 

Depending on whether $D(\textbf{x})$ is positive, for the mixed strategy set $\Delta(\mathcal{I})$, define
\begin{equation*}
    Z_p \triangleq \{\textbf{x}\in \Delta(\mathcal{I}): D(\textbf{x})\geq 0\} \quad \text{and}\quad Z_n \triangleq \{\textbf{x}\in \Delta(\mathcal{I}): D(\textbf{x})< 0\}.
\end{equation*}
Easy to see that the NE strategy $\textbf{x}^* \in Z_p$, hence $Z_p \neq \emptyset$.

To locate the regions $Z_p$ and $Z_n$ on $\Delta(\mathcal{I})$, we need the following inequality in probability theory.
\begin{lemma}[Lemma~A.1 of {\cite{cesa2006prediction}}]
\label{Lem_Hoffeding_inequility}
    Let $\mathbf{z}$ be a random variable with $a\leq \mathbf{z}\leq b$. Then, for any $s\in \mathbb{R}$, 
    \begin{equation*}
        \ln \mathbb{E}(e^{s\mathbf{z}})\leq s\mathbb{E}\mathbf{z} + \frac{s^2 (b-a)^2}{8}.
    \end{equation*}
\end{lemma}

Applying Lemma \ref{Lem_Hoffeding_inequility} , we can estimate $D(\textbf{x})$ as below.
\begin{align*}
    D(\textbf{x}) = \ln \left(  \sum\limits_{i=1}^n x_{i}  e^{\eta (v^\ast - e_i^T A \textbf{y}_x)} \right) &\leq \sum\limits_{i=1}^n x_i  \eta (v^\ast - e_i^T A \textbf{y}_x) + \frac{\eta^2 \delta^2}{8} \\
    & = \eta( v^\ast - x^T A \textbf{y}_x + \frac{\eta \delta^2}{8})
\end{align*}
where $\delta = \max\nolimits_{i\in \mathcal{I}, j\in \mathcal{J}}a_{i,j} - \min\nolimits_{i\in \mathcal{I}, j\in \mathcal{J}} a_{i,j}$. 
Then, depending on whether $v^\ast - \textbf{x}^T A \textbf{y}_x + \frac{\eta \delta^2}{8}$ is positive, we further split the set $\Delta(\mathcal{I})$ into another two regions $Z_u$ and $Z_d$ where 
\begin{equation*}
    Z_u \triangleq \{\textbf{x}\in \Delta(\mathcal{I}): v^\ast - \textbf{x}^T A \textbf{y}_x + \frac{\eta \delta^2}{8} \geq 0\} \quad \text{and}\quad Z_d \triangleq \{\textbf{x}\in \Delta(\mathcal{I}): v^\ast - \textbf{x}^T A \textbf{y}_x + \frac{\eta \delta^2}{8} < 0\}.
\end{equation*}
For the strategy $\textbf{x} \in \Delta(\mathcal{I}) $, if $ \textbf{x} \in Z_d$, immediately we have $D(\textbf{x}) < 0$, i.e., $ \textbf{x} \in Z_u$. Hence, we have the \textbf{Claim 1} below:

\textbf{Claim 1:} $Z_d\subset Z_n$ and $Z_p\subset Z_u$.

\begin{figure}[ht]
    \centering
    \includegraphics[width=0.8\textwidth]{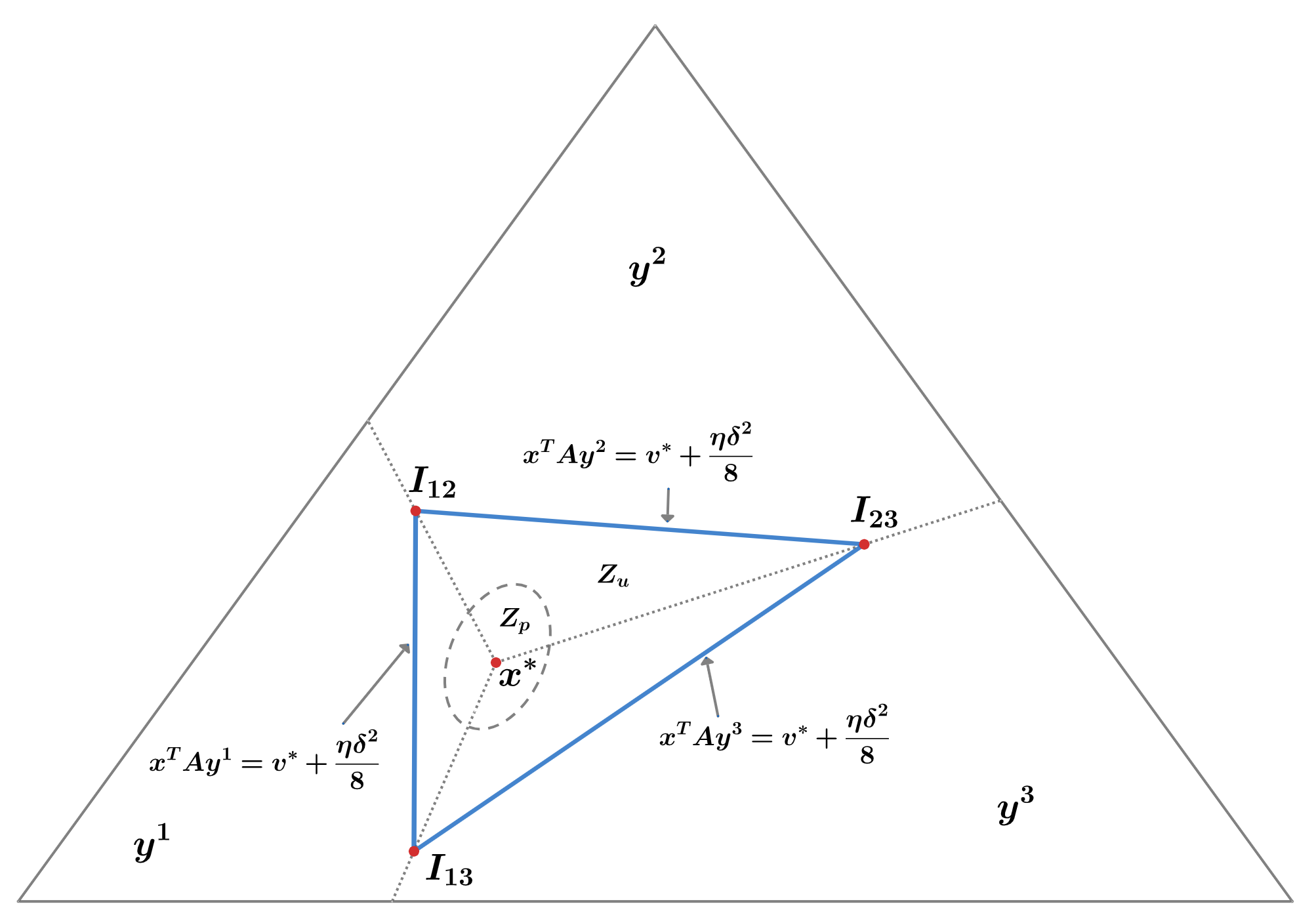}
    \caption{Graphical illustration of $Z_p$, $Z_u$ and NE for $3\times 3$ games.}
    \label{fig_illus}
\end{figure}

Figure \ref{fig_illus} gives a graphical illustration about the region $Z_p$, $Z_u$, $\Delta(\mathcal{I}) $ and interior NE strategy for a $3\times 3$ game. 
In this figure, the entire triangular region represents the simplex $\Delta(\mathcal{I})$, and the red point in the center represents the NE strategy $\textbf{x}^\ast$ of player X; the gray dashed lines divide the triangular region into three areas, labeled $\textbf{y}^1, \textbf{y}^2$, and $\textbf{y}^3$ respectively, which represents the best response action of player Y to the strategies in that area; the points on the blue solid line represent the strategy whereby the payoff of player Y is equal to $v^\ast + \frac{\eta \delta^2}{8}$ when player Y adopts the corresponding best response, and the triangular region enclosed by the blue solid line is region $Z_u$; within region $Z_u$, the elliptical region enclosed by the gray dashed line (the actual region $Z_p$ may not necessarily be elliptical) is region $Z_p$, and the NE strategy $\textbf{x}^\ast$ is located in region $Z_p$.

%Define 
%\begin{equation}\label{Def_f(x)}
 %   f(\textbf{x}) := \textbf{x}^T A \textbf{y}_x = \max\limits_{\textbf{y}\in %\Delta(\mathcal{J})} \textbf{x}^T A \textbf{y}.
%\end{equation}
%Then, the region $Z_u$ refers to the set of strategies $\textbf{x}$ satisfying %$f(\textbf{x}) \leq  v^\ast + \frac{\eta \delta^2}{8}$ and we have the claim below.

For the region $Z_u = \{ \textbf{x} \in \Delta(\mathcal{I}): \textbf{x}^T A \textbf{y}_x \leq v^\ast + \frac{\eta \delta^2}{8}\}$, we claim that:

\textbf{Claim 2: }The region $Z_u$ can be further rewritten as $Z_u = \{\textbf{x}\in \Delta(\mathcal{I}): \max\nolimits_{\textbf{y}^j} \textbf{x}^T A \textbf{y}^j \leq v^\ast + \frac{\eta \delta^2}{8}\} =  \{\textbf{x}\in \Delta(\mathcal{I}): A^T \textbf{x}\leq \textbf{b}\}$, where $\textbf{b} = [v^\ast + \frac{\eta \delta^2}{8}, v^\ast + \frac{\eta \delta^2}{8}, \cdots, v^\ast + \frac{\eta \delta^2}{8}]^T$. Thus, the region $Z_u$ is a bounded polyhedron.

For the bounded polyhedron, we have the following result, which can be found in the Theorem~2.9 of \cite{bertsimas1997introduction}.

\begin{lemma}[Representation of Bounded Polyhedra]\label{Lem_polyhedra}
    A bounded polyhedron is the set of all convex combinations of its vertices. 
\end{lemma}

In Figure \ref{fig_illus}, the points $I_{12}, I_{13}$ and $ I_{23}$ are the vertices of polyhedron $Z_u$. Thus by Lemma \ref{Lem_polyhedra}, every point in $Z_u$ can be written as a convex combination of these vertex points.
Then, we can prove that the polyhedron $Z_u$ must lie in the strict interior of $\Delta(\mathcal{I})$.

\begin{proposition}\label{Lem_interior}
    If $\textbf{x} \in Z_u$, then the elements of $\textbf{x}$ are uniformly lower bounded. That is, for all $\mathbf{x}\in Z_u$, there exists $\varepsilon_d>0$ such that $x_i > \varepsilon_d$ for all $i\in \mathcal{I}$. 
\end{proposition}
(Proof in Appendix \ref{Adx_lem_interior})

Based on Proposition \ref{Lem_interior}, we can obtain the following corollary.
\begin{corollary}\label{Lem_zp_bounded}
    In the Hedge-myopic system, if $\textbf{x} \in Z_u$, then $Q(\mathbf{x})$ is upper bounded, i.e., there exists $M_p > 0$ such that $Q(\mathbf{x}) \leq M_p$ for all $\mathbf{x}\in Z_u$.
\end{corollary}
\begin{proof}
    By Proposition \ref{Lem_interior}, we know that for all $\textbf{x}\in Z_u$, 
    \begin{equation*}
        Q(\textbf{x}) = -\sum\limits_{i=1}^n x_i^\ast \ln x_i \leq -\sum\limits_{i=1}^n x_i^\ast \ln \varepsilon_d = -\ln \varepsilon_d,
    \end{equation*}
    i.e., the function $Q(\textbf{x})$ is upper bounded in the region $Z_u$. 
  \end{proof}  
  
Since $Z_p \subset Z_u$, we immediately prove that $Q(\textbf{x})$ is also upper bounded by $M_p$ in the region $Z_p$. 

Corollary \ref{Lem_zp_bounded} indicates the boundness of $Q(\textbf{x})$ over $Z_p$ and $Z_u$, which is a property in a spatial sense. The subsequent proof of Theorem \ref{Lem_Qsequence_bounded}  demonstrates that the boundness in the spatial sense actually implies boundness in the temporal sense. 

\begin{proof}[Proof of Theorem \ref{Lem_Qsequence_bounded}]
    By Corollary \ref{Lem_zp_bounded}, if $\textbf{x} \in Z_u $, then $Q(x)\leq M_p $.
    
    Consider the sequence $\textbf{x}_t$ in the Hedge-myopic system.  Suppose that for some time $t$, $\textbf{x}_t\in Z_p$. Then, 
    \begin{align}
        Q(\textbf{x}_{t+1}) &= Q(\textbf{x}_t) + \ln \left(  \sum\limits_{i=1}^n x_{i, t}  e^{\eta (v^\ast - e_i^T A \textbf{y}_t)} \right) \nonumber \\
        & \leq M_p +  \ln \left(  \sum\limits_{i=1}^n x_{i, t}  e^{\eta \max\limits_{i\in\mathcal{I}} (v^\ast - e_i^T A \textbf{y}_t)} \right) \nonumber \\
        & = M_p + \eta \max\limits_{i\in\mathcal{I}} (v^\ast - e_i^T A \textbf{y}_t)  \nonumber \\
        & \leq M_p + \eta \delta_u, \label{budengshiQ}
    \end{align}
    where $\delta_u = v^\ast - \min\nolimits_{i\in\mathcal{I}, j\in \mathcal{J}} a_{ij}.$
    
    Now, we consider different cases for the behavior of the strategy sequence $\{\textbf{x}_t\}$. 
    
    \textbf{Case 1:} if the strategy sequence never goes into the region $Z_p$, which implies that $D(\textbf{x}_t) < 0$ for all $t$, then we have $Q(\textbf{x}_t) \leq Q(\textbf{x}_1)$ for all $t$. By calculating, $Q(\textbf{x}_1) = -\sum\nolimits_{i=1}^n x_i^\ast \ln (1/n) = \ln n$. Hence, $Q(\textbf{x}_t) \leq \ln n$ for all $t$.
    
    \textbf{Case 2:} if the strategy sequence goes into the region $Z_p$ at some time $t^\prime$, then for the strategy $\textbf{x}_t$ before time $t^\prime$, we have $Q(\textbf{x}_t) \leq \ln n$ for all $t < t^\prime$. For the strategy $\textbf{x}_t$ after time $t^\prime$:\\ 
    (1). if $\textbf{x}_t\in Z_p$, then we have $Q(\textbf{x}_t) \leq M_p$; \\
    (2). if $\textbf{x}_t \notin Z_p$, then we can find an integer $k > 0$ such that $\textbf{x}_{t-k}\in Z_p$ and $\textbf{x}_{t-j} \notin Z_p$ for $0< j < k$. 
    Combining this with \eqref{budengshiQ}, we can obtain that $Q(\textbf{x}_t) < Q(\textbf{x}_{t-1}) < \cdots < Q_{t-k+1} \leq M_p + \eta \delta_u$ since $\textbf{x}_{t-k}\in Z_p$. 
 
    Take $M_Q = \max \{\ln n, M_p + \eta \delta_u\}$, then the Q-sequence $\{Q_t\}_{t=1}^\infty$ is upper bounded by $M_Q$. 
\end{proof}

From Definition \eqref{Def_Q(x)}, for strategy $\textbf{x}$, if some element $x_{i}$ is near zero, the value of $Q(x)$ is near infinity. Hence, by Theorem \ref{Lem_Qsequence_bounded}, the elements of the stage strategy $\textbf{x}_t$ cannot be too small since the Q-sequence is bounded. Based on this direct intuition, we can further prove the following theorem, which shrinks the range of possible values for $\textbf{x}_t$ to be a finite set. 

\begin{theorem}\label{Lem_finite_values}
     In the Hedge-myopic system, $\textbf{x}_t$ for player X can only take finite values.
\end{theorem}
(Proof in Appendix \ref{Appendix_C})

% 从定理 \ref{Lem_finite_values}的证明中可以发现，当收益矩阵元素有无理数时，除非所有元素是同一无理数的有理数倍的特殊情况，该定理将不再成立。这也说明了，收益矩阵元素都是有理数这一条件几乎是周期性成立的必要条件。

\begin{remark}
    From the proof of Theorem \ref{Lem_finite_values}, it can be observed that when there are irrational number elements in the payoff matrix, Theorem \ref{Lem_finite_values} no longer holds except for the special case such as all elements are rational multiples of the same irrational number. This implies that the rationality of the payoff matrix is nearly an intrinsically necessary condition for periodicity.
\end{remark}

By Theorem \ref{Lem_finite_values}, in the Hedge-myopic system, player X can only adopt a finite number of mixed strategies. This directly leads to the periodicity of the dynamical system as stated below. 

\begin{theorem}\label{Thm_periodic}
    In the Hedge-myopic system with \textbf{Assumption 1} and \textbf{Assumption 2} satisfied, after finite steps, 
    \begin{enumerate}[(1)]
        \item the strategy sequence of player X and player Y enters a cycle (i.e., is periodic);
        \item the time-averaged strategy of player Y in one period is a NE strategy.
    \end{enumerate}
    In other words, there exists $T_s$ and $T$ such that for all $t\geq T_s$, we have $\mathbf{x}_{t+T} = \mathbf{x}_t, \mathbf{y}_{t+T} = \mathbf{y}_t,$ and $\sum\nolimits_{k=1}^T \mathbf{y}_{t+k} / T = \mathbf{y}^\ast$, where $\mathbf{y}^\ast$ is a NE strategy of player Y.
\end{theorem}

Below we give an example to illustrate Theorem \ref{Thm_periodic}.

\begin{figure}
    \centering
    \includegraphics[width = \textwidth]{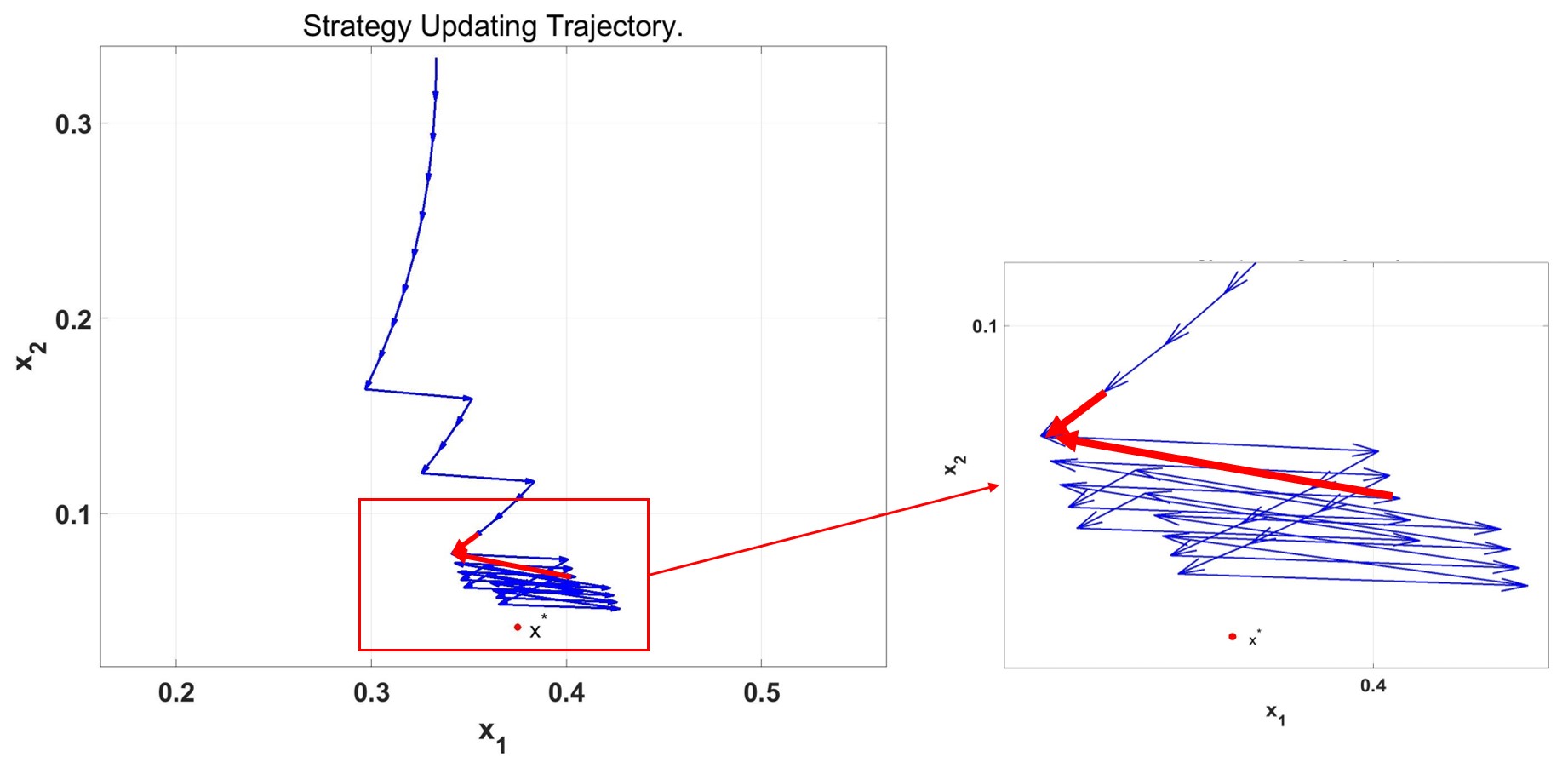}
    \caption{The evolving path of the strategy for player X. The red point represents the NE strategy of player X.}
    \label{fig_evolving_path}
\end{figure}

\begin{example}\label{ex_theorem_explain}
Consider a $3\times 3$ zero-sum game and the payoff matrix is taken as 
    \begin{equation*}
        A = \begin{pmatrix}
            -2 & 1 & 3 \\
            1 & 2 & -2 \\
            2 & 0 & -1
        \end{pmatrix},
    \end{equation*}
    and $\eta = \frac{\ln3}{625}$. In the Hedge-myopic system for this game, the evolution of $\textbf{x}_t$ is shown in Figure \ref{fig_evolving_path}. The X-axis represents the first element $x_{1,t}$, and the Y-axis represents the second element $x_{2,t}$. 
    
    From Figure \ref{fig_evolving_path}, we can observe that basically $\textbf{x}_t$ gradually approaches the NE strategy of player X. However, after reaching a certain range, $\textbf{x}_t$ enters a cycle. In Figure \ref{fig_evolving_path}, we mark a point such that two red arrows both point to it, meaning that a cycle is formed. Additionally, we can see that $\textbf{x}_t$ does not converge to his NE strategy, no matter how long the game is repeated.
\end{example}

Before proving Theorem \ref{Thm_periodic}, we give Lemma \ref{Lem_X_cycle_to_Y_cycle} below by which we only need to prove the periodicity of $\textbf{x}_t$ in order to prove the periodicity of the system. 

\begin{lemma}\label{Lem_X_cycle_to_Y_cycle}
    If $\textbf{x}_t$ enters a cycle, then $\textbf{y}_t$ also enters a cycle and her time-averaged strategy in a single cycle is a NE strategy.
\end{lemma}
(Proof in appendix \ref{Appendix_A})

Now, we can prove Theorem \ref{Thm_periodic}.
\begin{proof}[Proof of Theorem \ref{Thm_periodic}]
    By Lemma \ref{Lem_finite_values}, $\textbf{x}_t$ can only take finite values. Then, using the pigeonhole principle, we can obtain that there must exist two stages, $t_1 < t_2$, such that $\textbf{x}_{t_1} = \textbf{x}_{t_2}$ because the game is repeated for infinite times. Since $\textbf{x}_{t+1}$ is fully determined by $\textbf{x}_t$, we have $\textbf{x}_{t_1+k} = \textbf{x}_{t_2+k}, \ \forall\ k = 1, 2, \cdots$, which implies that $\textbf{x}_t$ is periodic from time $t_1$. Combining Lemma \ref{Lem_X_cycle_to_Y_cycle}, we can prove Theorem \ref{Thm_periodic}.
\end{proof}

%\begin{remark}
 %   Periodicity is a stronger result compared to Poincar\'e recurrence, which is %determined by the setting of the game system. In the adversarial regularized learning %scenarios, the stochastic nature of learning algorithms prevents the attainment of %periodicity.
%\end{remark}

\begin{remark}
    The time required to enter a cycle depends on the parameter $\eta$. Due to the complexity of the problem, it is difficult to provide an explicit expression for it. We have studied it for $2\times 2$ games in \cite{guo2023taking} where the time needed for the strategy sequences of both players to enter a cycle is $O(1/\eta)$.  
\end{remark}

\begin{remark}
    In addition to periodicity, it is worth noting that the time-averaged strategy of player Y in a single cycle is a \textbf{precise} NE strategy! Compared with this, when both players adopt the no-regret learning algorithm in two-player zero-sum games, their time-averaged strategy profile converges to a NE when the time horizon goes to infinity, that is to say, only an approximate NE can be obtained. Moreover, in the Hedge-myopic system, not only can we obtain a precise NE, but we also only need to compute the time-averaged strategy in a single cycle whose length is far shorter than the whole time horizon.
\end{remark}

% In the proof of Theorem \ref{Thm_periodic}, by utilizing the pigeonhole principle, we cleverly prove that strategy sequences for player X and Y are periodic. The premise for the theorem to hold is that the elements of the payoff matrix are all rational numbers. 

\subsection{Non-periodicity Examples without Rational Interior NE}

In this section, we study the dynamic of the Hedge-myopic system for general games by examples. These games may have an unique irrational interior equilibrium, or have more than one interior equilibria, or have an unique non-interior equilibrium. We will see that the dynamics of the Hedge-myopic systems can vary significantly for different games, preventing a consistent result. 

First, we give Example \ref{ex_irrational_case} to show that when the NE of the game is irrational, the periodicity of the strategy sequence no longer holds.

\begin{figure}[htbp]
  \centering
  \begin{subfigure}[b]{0.95\textwidth}
    \includegraphics[width=\textwidth]{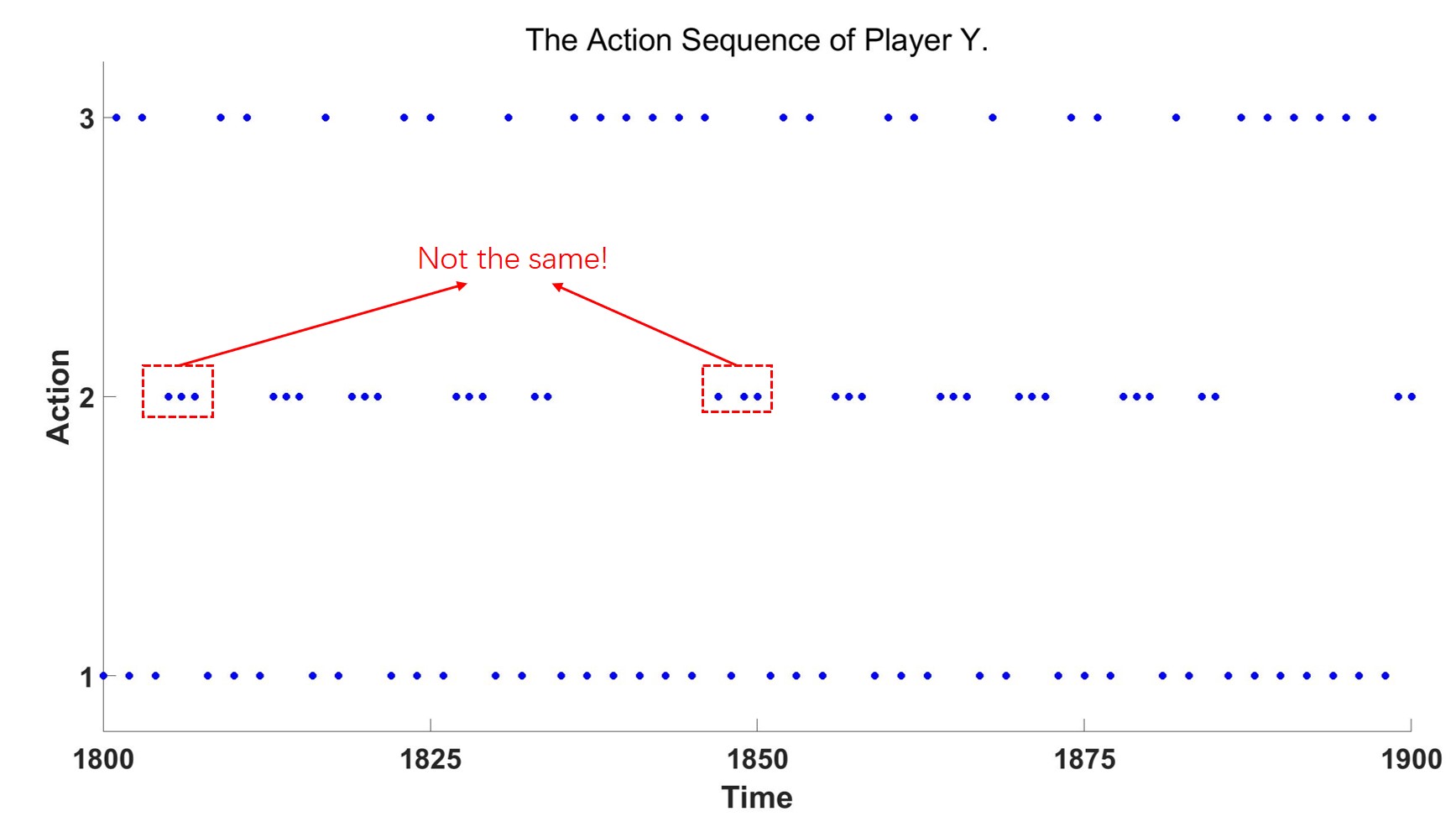}
    \caption{The action sequence is not periodic when there are irrational numbers in the payoff matrix.}
    \label{fig_irrational_case}
  \end{subfigure}
  \begin{subfigure}[b]{0.95\textwidth}
    \includegraphics[width=\textwidth]{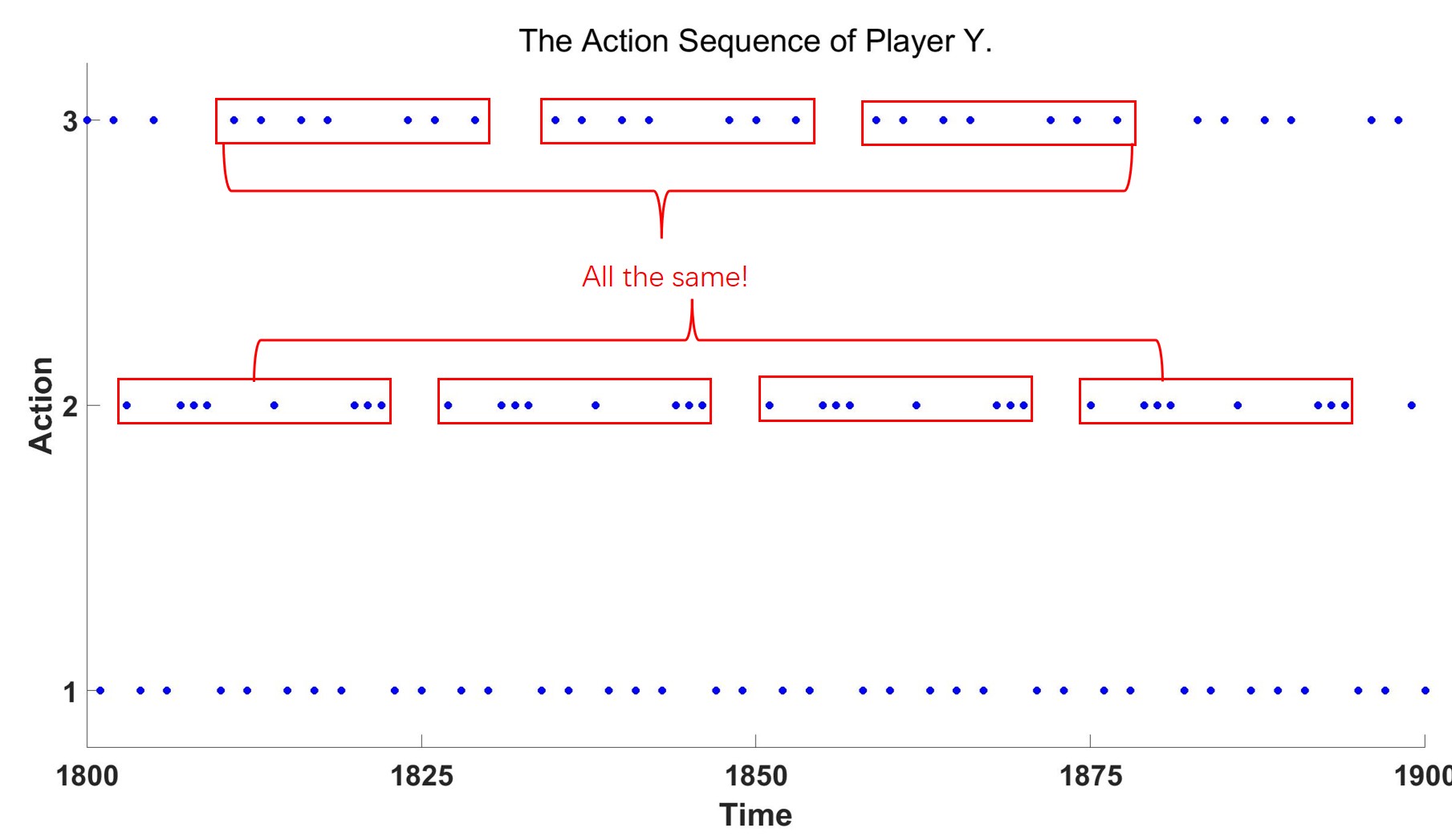}
    \caption{The action sequence is periodic when the elements of payoff matrix are all rational numbers.}
    \label{fig_rational_case}
  \end{subfigure}
  \caption{The action sequence of player Y in two different games.}
  \label{fig_not_periodic_example}
\end{figure}

\begin{example}\label{ex_irrational_case}
    Consider a $3\times 3$ zero-sum game with the payoff matrix for player Y being 
    \begin{equation*}
        \begin{pmatrix}
            -2 & \sqrt{2} & 3 \\
            1 & 2 & -2 \\
            2 & 0 & -1
        \end{pmatrix},
    \end{equation*}
    in which there is an irrational number $\sqrt{2}$. For this game, in the Hedge-myopic system, the action sequence $\textbf{y}_t$ is shown by Figure \ref{fig_irrational_case}. We only present the part that lies between stage 1800 and 1900. 
    
    For comparison,  we also show the action sequence $\textbf{y}_t$ for the game in Example \ref{ex_theorem_explain} by Figure \ref{fig_rational_case}. Comparing these two figures, we can observe that the action sequence is periodic when the elements of payoff matrix are all rational numbers, while the action sequence is not periodic when there are irrational numbers in the elements of payoff matrix.
\end{example}

% 从引理\ref{Lem_X_cycle_to_Y_cycle}的证明中可以看出，参与人X的阶段策略进入循环的必要条件是参与人Y在一个周期内的时间平均策略是其纳什均衡策略。由于参与人Y只能选择纯策略，其时间平均策略元素都是有理数，所以当纳什均衡策略元素有无理数时，其时间平均策略不可能是纳什均衡策略。

%From the proof of Lemma \ref{Lem_X_cycle_to_Y_cycle}, it can be seen that a necessary %condition for the stage strategy of player X to enter a cycle is that the averaged %strategy of player Y over one cycle is her NE strategy. 

In the Hedge-myopic system, since player Y can only take pure strategies, the averaged strategy has only rational elements, so it is impossible for the averaged strategy to be a NE strategy when there are irrational numbers in the NE strategy. This explains why the periodicity no longer holds. 

Below we present Example \ref{Exp_non_interior} and Example \ref{Exp_special_case} to illustrate that when the game does not admit interior equilibrium, the periodicity of system dynamics also no longer holds, and periodicity only exists in special games. 

\begin{example}[No interior equilibrium and no cycle]\label{Exp_non_interior}
    Consider a $3 \times 3$ zero-sum game and let the payoff matrix of player Y be
    \begin{equation*} A = 
        \begin{pmatrix}
            -2 & 1 & 3 \\
            1 & -1 & 2 \\
            0 & 2 & -1
        \end{pmatrix}.
    \end{equation*}
    The NE strategy for player X is $(0, 1/2, 1/2)$, which is non-interior, and the equilibrium strategy for player Y is not unique. The evolving path of $\textbf{x}_t$ is shown by Figure \ref{fig_non_interior_case}, from which we can see that the first element $x_{1,t}$ continually decreases and the repeated game does not enter a cycle. 
\end{example}

\begin{figure}[htbp]
  \centering
  \includegraphics[width=\textwidth]{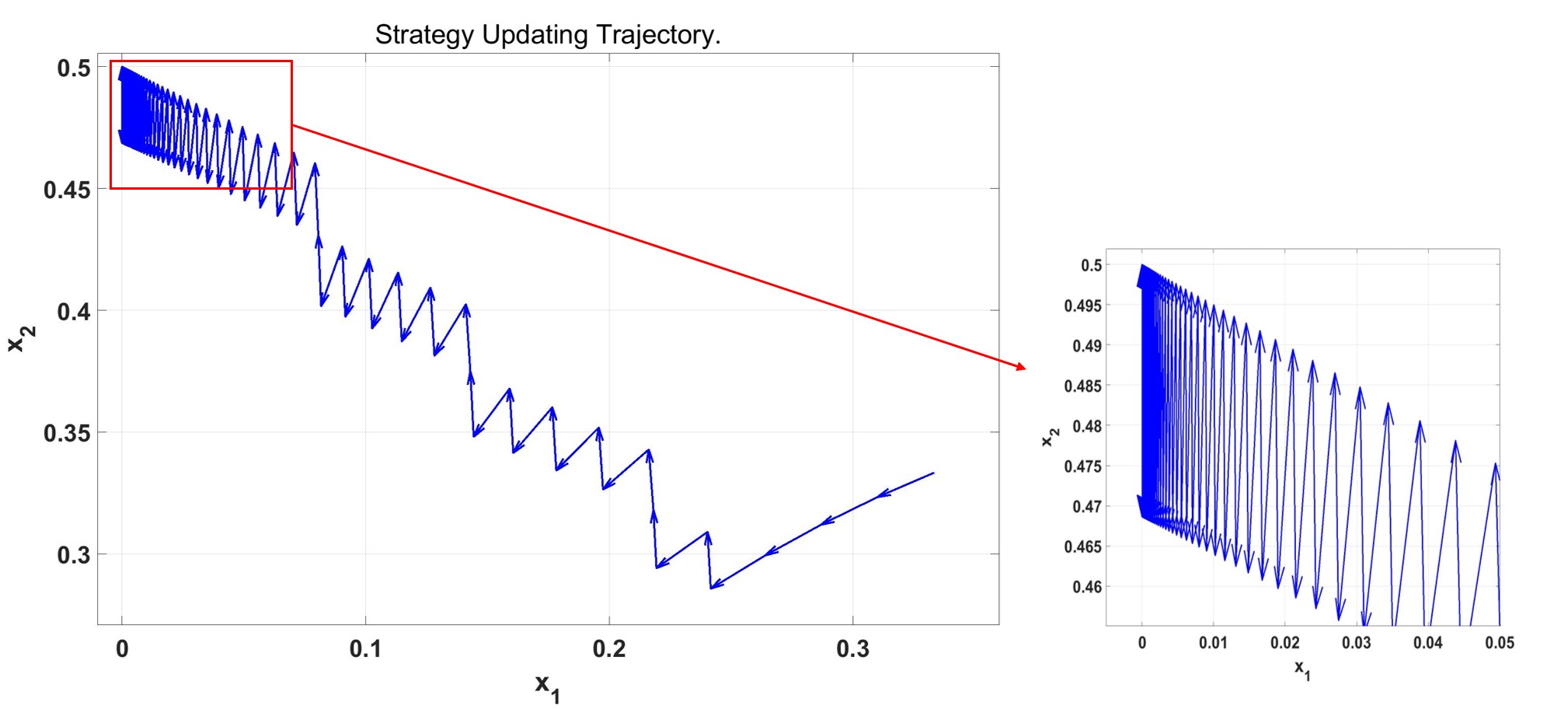}
  \caption{The evolution of $\textbf{x}_t$ in Example \ref{Exp_non_interior}.}
  \label{fig_non_interior_case}
\end{figure}

\begin{example}[No interior equilibrium but with a special cycle]\label{Exp_special_case}
    Consider the $3 \times 3$ game with a payoff matrix being
    \begin{equation*} A = 
        \begin{pmatrix}
            2 & -1 & 0 \\
            -1 & 1 & -2 \\
            -3 & -2 & 1
        \end{pmatrix}.
    \end{equation*}
    In this game, player X has both interior and non-interior NE strategy. The evolving path of $x_t$ is shown in Figure \ref{fig_special_case}, from which we can see that the repeated game does enter a cycle although the assumptions in Theorem \ref{Thm_periodic} does not hold. 
\end{example}

\begin{figure}[htbp]
  \centering
  \begin{subfigure}[b]{0.45\textwidth}
    \includegraphics[height=0.96\textwidth, width=\textwidth]{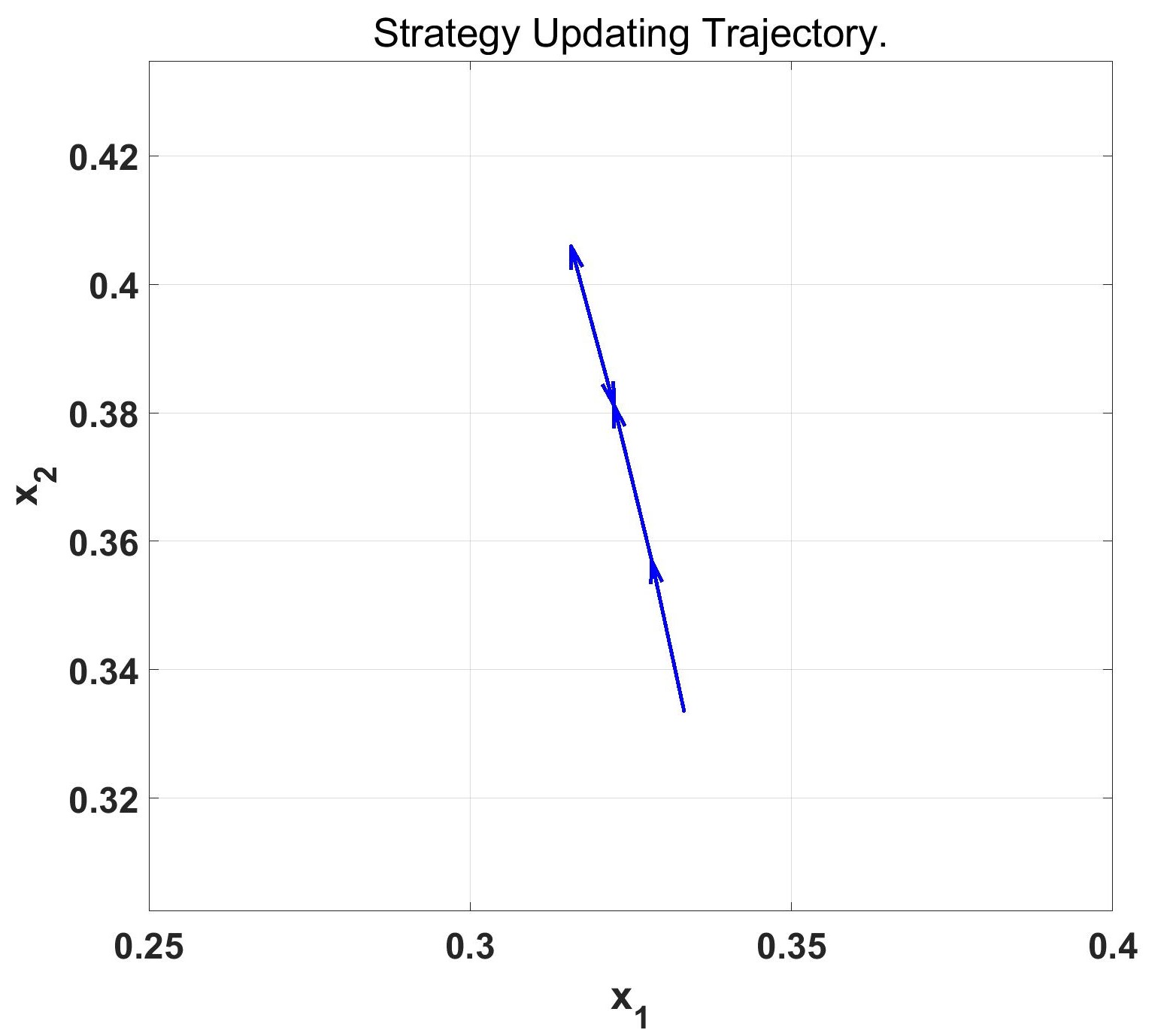}
    \caption{The evolution of $x_{1,t}$.}
    \label{fig_special_case_a}
  \end{subfigure}
  \begin{subfigure}[b]{0.45\textwidth}
    \includegraphics[height=0.99\textwidth, width=\textwidth]{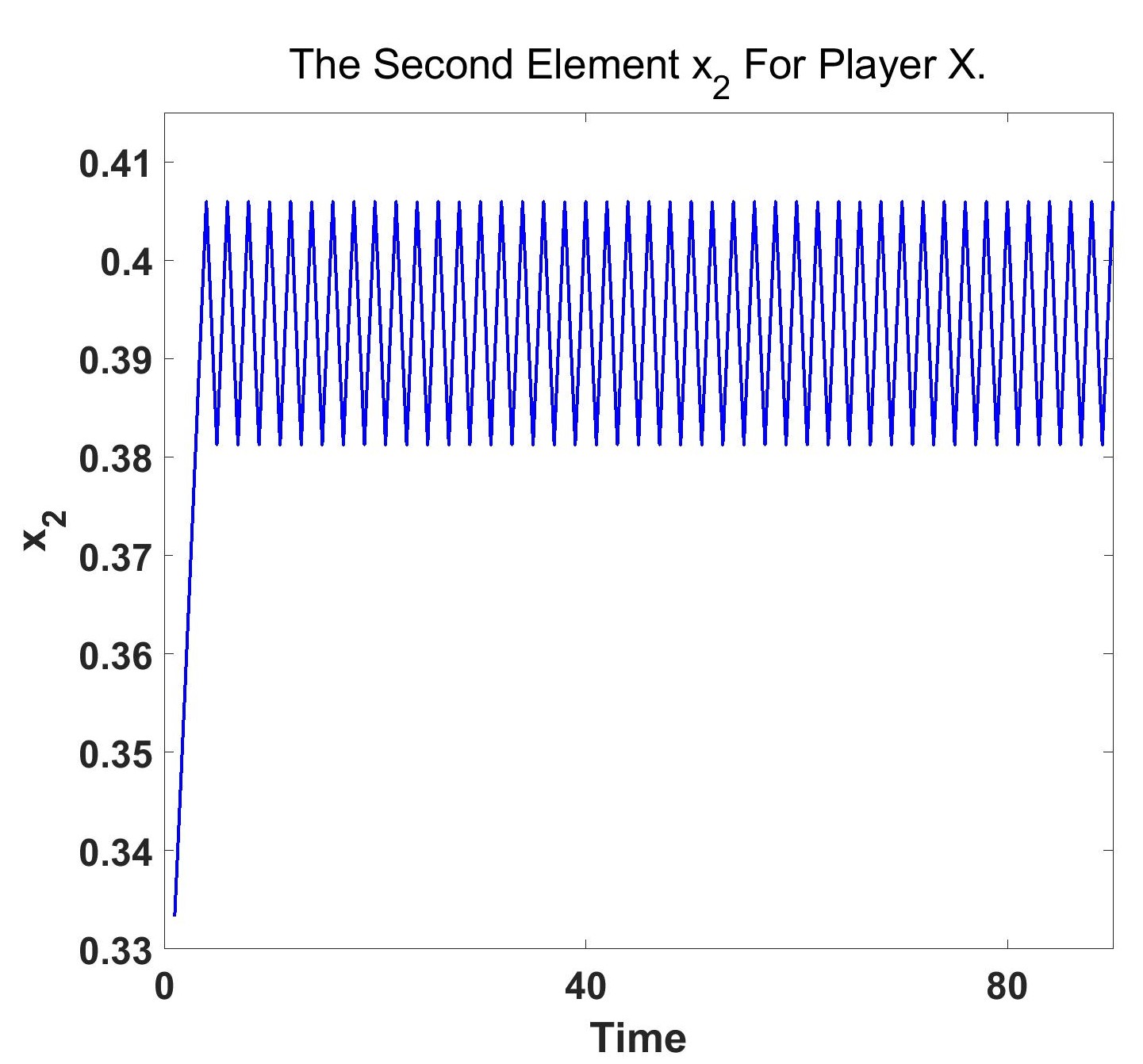}
    \caption{The evolution of $x_{2, t}$.}
    \label{fig_special_case_b}
  \end{subfigure}
  \caption{The evolution of $x_t$ in the Hedge-myopic system for Example \ref{Exp_special_case}.}
  \label{fig_special_case}
\end{figure}

Generally, for all the two-player games, we can approximate the NE no matter whether the Hedge-myopic system enters a cycle or not. We state it in Theorem \ref{Lem_HBR_convergence} below, whose proof can also be found in \cite{Li2023-gd}.

\begin{theorem}\label{Lem_HBR_convergence}
    In a repeated two-player zero-sum game with $T$ stages, suppose player X employs the Hedge algorithm to update his stage strategy with parameter $\eta$ set to be $\sqrt{8\ln n / T}$, where $n$ is the number of his feasible actions, and suppose player Y takes the myopic best response to the stage strategy of player X, then the time-averaged strategy profile converges to the $\sqrt{\ln n/(2T)}-$NE. 
\end{theorem}
% (Proof in Appendix \ref{Appendix_D})

\begin{proof}
    By Theorem \ref{thm_regret_bound}, we know that the upper regret bound for the Hedge algorithm with $\eta = \sqrt{8\ln n/ T}$ is $\sqrt{(T/2)\ln n}$. Then, we have \begin{equation*}
        \frac{1}{T}\sum\limits_{t=1}^T \textbf{x}_t^T A \textbf{y}_t - \min\limits_{\textbf{x}\in \Delta(\mathcal{I})} \textbf{x}^T A \bar{\textbf{y}} = 
        \frac{1}{T}\left( \max\limits_{\textbf{x}\in \Delta(\mathcal{I})} \sum\limits_{t=1}^T \textbf{x}^T (-A) \textbf{y}_t -  \sum\limits_{t=1}^T \textbf{x}_t^T (-A) \textbf{y}_t \right)  \leq \sqrt{\frac{\ln n}{2T}},
    \end{equation*}
    where $\Bar{\textbf{y}} = \sum\nolimits_{t=1}^T (\textbf{y}_t / T)$. Thus, for all $\textbf{x}\in \Delta(\mathcal{I})$, \begin{equation}\label{Eq_Appendix_D_eq1}
        \frac{1}{T}\sum\limits_{t=1}^T \textbf{x}_t^T A \textbf{y}_t - \textbf{x}^T A \bar{\textbf{y}} \leq \sqrt{\frac{\ln n}{2T}}.
    \end{equation}

    Since player Y always takes the myopic best response, we can deduce that for all $\textbf{y}\in \Delta(\mathcal{J})$,
    \begin{equation}\label{Eq_Appendix_D_eq2}
         \bar{\textbf{x}}^T A \textbf{y} -  \frac{1}{T}\left( \sum\limits_{t=1}^T \textbf{x}_t^T A \textbf{y}_t \right) = \frac{1}{T}\left( \sum\limits_{t=1}^T \textbf{x}_t^T A \textbf{y} -  \sum\limits_{t=1}^T \textbf{x}_t^T A \textbf{y}_t \right)  = \frac{1}{T} \sum\limits_{t=1}^T \left(\textbf{x}_t^T A \textbf{y} - \textbf{x}_t^T A \textbf{y}_t\right)  \leq 0,
    \end{equation}
    where $\Bar{\textbf{x}} = \sum\nolimits_{t=1}^T (\textbf{x}_t / T)$. Combining the two inequalities \eqref{Eq_Appendix_D_eq1} and \eqref{Eq_Appendix_D_eq2}, we have 
    \begin{equation}
        \bar{\textbf{x}}^T A \textbf{y} - \textbf{x}^T A \bar{\textbf{y}} \leq \sqrt{\frac{\ln n}{2T}}
    \end{equation}
    for all $\textbf{x}\in \Delta(\mathcal{I})$ and $\textbf{y}\in \Delta(\mathcal{J})$. Taking $\textbf{x} = \bar{\textbf{x}}$, we obtain $\bar{\textbf{x}}^T A \textbf{y} - \bar{\textbf{x}}^T A \bar{\textbf{y}} \leq \sqrt{\ln n/2T}$ for all $\textbf{y}\in \Delta(\mathcal{J})$, while taking $\textbf{y} = \bar{\textbf{y}}$, we obtain that $\bar{\textbf{x}}^T A \bar{\textbf{y}} - \textbf{x}^T A \bar{\textbf{y}} \leq \sqrt{\ln n/2T}$ for all $\textbf{x}\in \Delta(\mathcal{I})$. Hence, the time-averaged strategy profile $(\bar{\textbf{x}}, \bar{\textbf{y}})$ forms a $\sqrt{\ln n/(2T)}-$NE.
\end{proof}

% From here, we can draw a conclusion that regardless of whether the payoff function is rational or irrational, whether the NE is an interior or non-interior equilibrium, the convergence to the equilibrium strategy for the time-averaged strategy always holds. However, the dynamics of the game systems can vary greatly for different cases, and studying their dynamic behaviors is quite challenging.

\section{A Novel NE-solving Algorithm and Experiments}\label{sec_HBR}

\subsection{HBR: an Asymmetric Equilibrium-solving Paradigm}

\begin{algorithm}[ht]
\renewcommand{\algorithmicrequire}{\textbf{Input: }}
\renewcommand{\algorithmicensure}{\textbf{Output: }}
\caption{HBR Paradigm}
\label{alg_HBR}
\begin{algorithmic}[1]
\setstretch{1.1}

\REQUIRE $T$: time horizon, $A$: payoff matrix, $n$: number of actions for the row player, $m$: number of actions for the column player

\STATE set $\eta = \sqrt{8\ln n/T}$ 
\STATE initialize $\textbf{c}_t = \textit{zeros}(m, 1)$ \COMMENT{$c_t$ is the count of each action taken by player Y}
\STATE initialize the time-averaged strategies: $\operatorname{TAS}_x = \textit{zeros}(n, 1),\ \operatorname{TAS}_y = \textit{zeros}(m, 1)$ 

\FOR{$t = 1$ to $T$}
    \STATE compute the state vector $s_t \gets A\cdot c_t$
    \STATE normalization: $\textbf{s}_t$ $\gets$ $\textbf{s}_t - \textbf{s}_t(n) \cdot \textit{ones}(n, 1)$ \COMMENT{$\textbf{s}_t(n)$ is the n-th element of $\textbf{s}_t$}
    \STATE check if there exists $t^\prime < t$ such that $\textbf{s}_{t^\prime} = \textbf{s}_t $
    \IF{such $t^\prime$ exists}   
        \STATE compute the average of $\textbf{y}_t$ between time $t^\prime$ and $t$
        \STATE output the average strategy \COMMENT{the average strategy is player Y's exact NE strategy}
        \STATE terminate the \textbf{for} loop
    \ELSE
    \STATE compute the stage strategy $\textbf{x}_t$ for player X based on Equation \eqref{Eq_hedge_formula}
    \STATE compute the best response $\textbf{y}_t$ for player Y based on Equation \eqref{Eq_yt_formula}
    \STATE update the time-averaged strategy $\operatorname{TAS}_x$ of player X:  $ \operatorname{TAS}_x = \frac{t-1}{t} \operatorname{TAS}_x + \frac{1}{t} \textbf{x}_t$ 
    \STATE update the time-averaged strategy $\operatorname{TAS}_y$ of player Y:  $ \operatorname{TAS}_y = \frac{t-1}{t} \operatorname{TAS}_y + \frac{1}{t} \textbf{y}_t$
    \STATE store $\textbf{s}_t$ and $\textbf{y}_t$ 
    \STATE update $\textbf{c}_t$ $\gets$ $\textbf{c}_t + \textbf{y}_t$
    \ENDIF
\ENDFOR

\IF{the \textbf{for} loop is not prematurely terminated}
     \STATE output $\operatorname{TAS}_x$ and $\operatorname{TAS}_y$
\ENDIF

\end{algorithmic}
\end{algorithm}

Based on Theorem \ref{Thm_periodic} and Theorem \ref{Lem_HBR_convergence}, we can propose an asymmetric Nash-equilibrium solving paradigm for two-player zero-sum games. Given the time horizon $T$ and the payoff matrix $A$, let one player employ the Hedge algorithm to update his stage strategy, and let the other player know all the information and take the myopic best response. The paradigm is called \textit{the HBR paradigm} in which \textit{H} stands for the Hedge algorithm of one player and \textit{BR} stands for the best response of the other player.

As we know from Theorem 3.3, we can compute exact NE fast if there is a cycle. So we exert additional efforts to identify whether the strategy sequences enter a cycle. If a cycle is detected, the computation can be terminated, and by Theorem \ref{Thm_periodic}, an exact NE strategy for the player using the myopic best response can be obtained. If the strategy sequences never enter a cycle, we can output the time-averaged strategy over the entire time horizon. By Theorem \ref{Lem_HBR_convergence}, this converges to an approximate NE as the time horizon T increases. By exchanging the updating rules of the players, an exact or approximate NE for the other player can also be obtained. 

So how can we identify a cycle? This can be done through $\{x_t\}$. To be specific, we can record all values of $x_t$, and once $x_{t_1} = x_{t_2}$, then the strategy sequence $x_{t_1}, x_{t_1+1}, \cdots, x_{t_2 - 1}$ forms a cycle. By the proof of Lemma \ref{Lem_X_cycle_to_Y_cycle} and \eqref{Eq_hedge_formula}, the vector 
$s_t = \operatorname{norm}(A\cdot \sum\nolimits_{\tau=1}^{t-1} y_\tau)$ is the other object to detect, where the operator $\operatorname{norm}$ means subtracting each element of a vector by its $n$-th element. Then it can be observed that the detection of a cycle implies that player Y has an interior equilibrium strategy.

The pseudocode for this paradigm is shown in Algorithm \ref{alg_HBR}.

\subsection{Experimental Results}\label{sec_experiments}

We conduct several experiments to show the effectiveness of the HBR paradigm. First, we consider the game whose payoff matrix is set to be
\begin{equation*}
    A = \begin{pmatrix}
    -2 & 1 & 3\\
    1 & 2 & -2 \\
    2 & 0 & -1
    \end{pmatrix}.
\end{equation*}
The equilibrium of this game is an interior equilibrium, which is $((\frac38, \frac{1}{24}, \frac{7}{12}), (\frac38, \frac13, \frac{7}{24}))$. An existing NE solving paradigm via learning is by the self-play of no-regret algorithm, such as Hedge, abbreviated to HSP. Now, we compare the performance of HBR and HSP paradigm.

\begin{figure}[htbp]
  \centering
  \begin{subfigure}[b]{0.45\textwidth}
    \includegraphics[width=\textwidth]{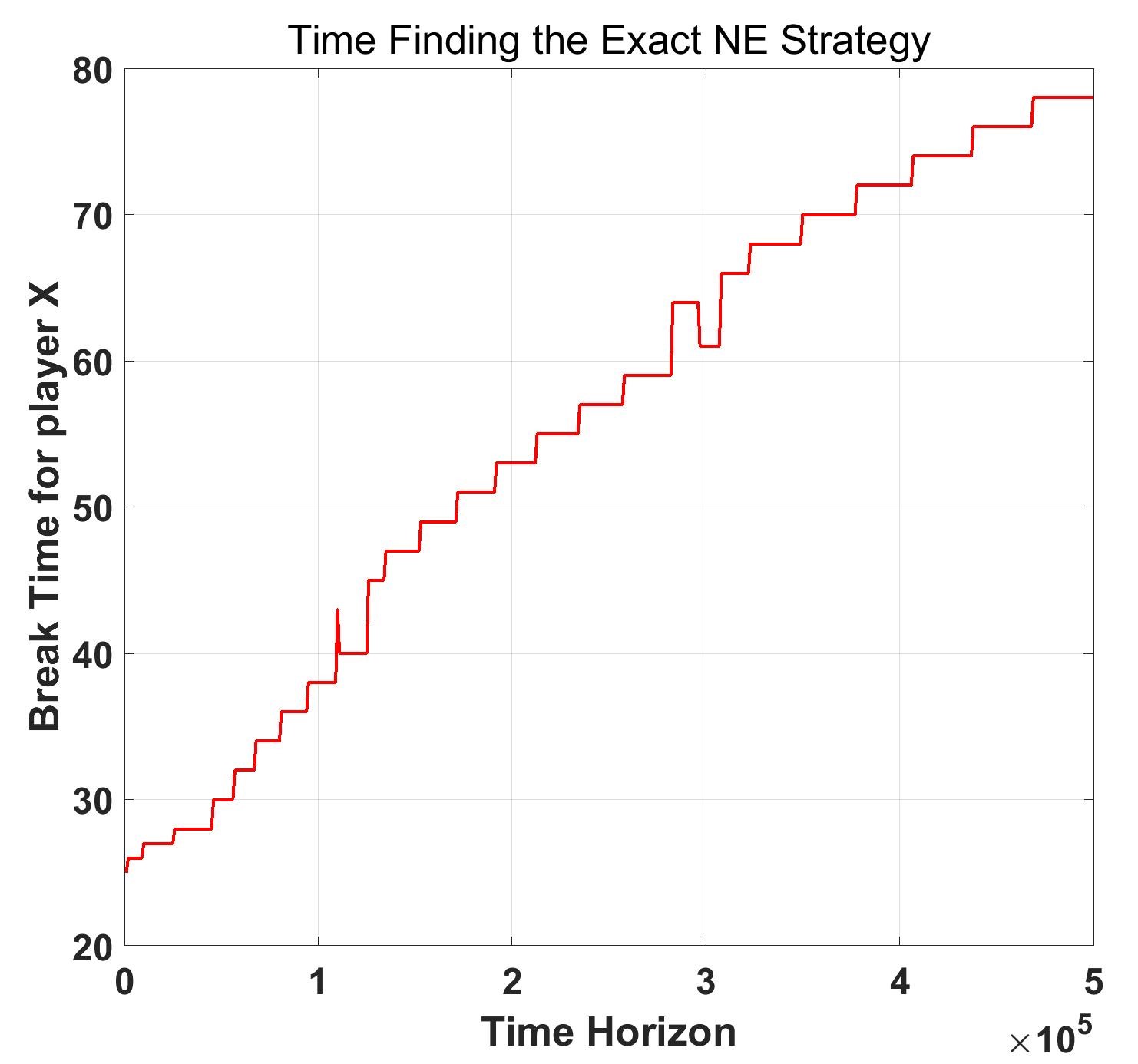}
    \caption{The relationship between the terminal time and the time horizon for computing the NE strategy for player X.}
    \label{fig_stop_time_X}
  \end{subfigure}
  \begin{subfigure}[b]{0.45\textwidth}
    \includegraphics[width=\textwidth]{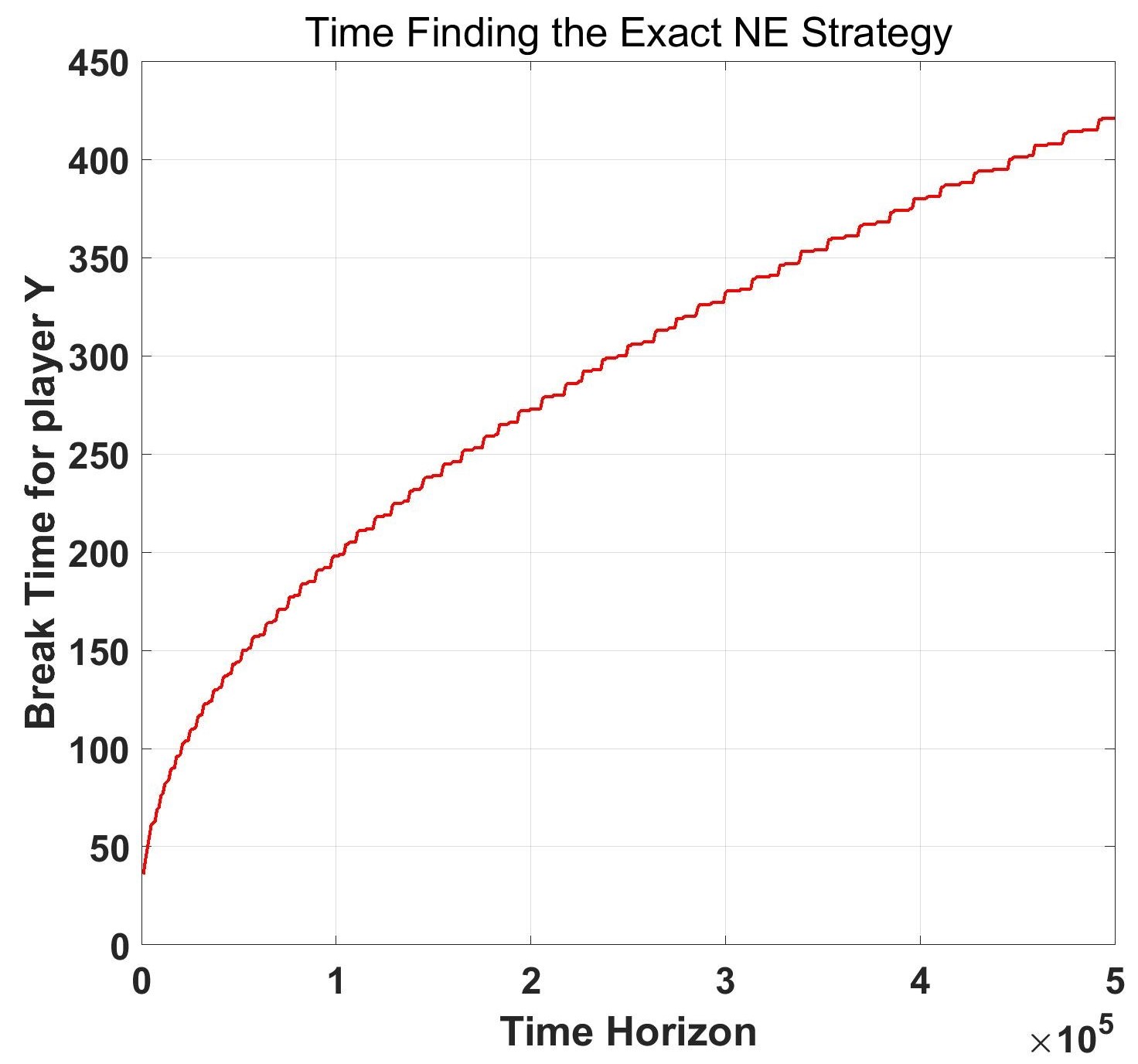}
    \caption{The relationship between the terminal time and the time horizon for computing the NE strategy for player Y.}
    \label{fig_stop_time_Y}
  \end{subfigure}
  \caption{Autonomously terminated time when the game admits an interior equilibrium.}
  \label{fig_stop_time}
\end{figure}

\begin{enumerate}[(1).]
    \item For different time horizons $T$, let $\eta = \sqrt{8\ln 3/ T}$. Figure \ref{fig_stop_time} gives the time needed to get the NE for player X and Y by the HBR paradigm. From Figure \ref{fig_stop_time}, we can see that the HBR paradigm actually enters a cycle very quickly. For this special game, time needed to calculate the NE strategy for player X and Y is different. Basically, needed steps to enter a cycle increases with $\eta$.

    For example, when $T=500000$, calculating the NE strategy for player Y terminates prematurely at time $t\approx 420$, while calculating the NE strategy for player X terminates prematurely at time $t\approx 80$. Basically, we can save a substantial amount of computation. 

    Why is the termination time so different for player X and Y? The reason lies in their NE strategies. For player X, it is $(\frac38, \frac{1}{24}, \frac{7}{12})$, while for player Y,  it is $(\frac38, \frac13, \frac{7}{24})$, which is more balanced among the elements. It is natural that it needs shorter time to enter a cycle corresponding to $(\frac38, \frac13, \frac{7}{24})$. 

    \item For different time horizons $T=1000, 2000, \cdots, 80000$, let $\eta = \sqrt{(8\ln 3)/ T}$. We calculate the $\operatorname{ND}$ of the time-averaged strategy profile and the results are in Figure \ref{fig_ND_compare}. This figure illustrates that, 1) as $T$ increases, the $\operatorname{ND}$ of the time-averaged strategy profile obtained by HBR roughly decreases, indicating that the calculated results approach the NE more closely; 2) compared to HSP, the volatility of the performance of HBR is lower. 
    
    We set $T=3000$ and present the $\operatorname{ND}$ of the time-averaged strategy at each time $t$ in Figure \ref{fig_ND_descending_trend}. We can see that compared to HSP, the ND of the averaged strategy of HBR decreases faster and fluctuates lighter once it stabilizes. 

    Furthermore, we consider another game matrix
    \begin{equation*}
        A = \begin{pmatrix}
        -2 & \sqrt{2} & 3\\
        1 & 2 & -2 \\
        2 & 0 & -1
    \end{pmatrix}.
\end{equation*}
    whose equilibrium is not rational interior. The $\operatorname{ND}$ of the time-averaged strategy profile and the convergence rate are shown in Figure \ref{fig_ND_compare_1} and \ref{fig_ND_descending_trend_1} respectively. In this case, the HBR paradigm still performs better than HSP in the perspective of stability and convergence rate.
\end{enumerate}

\begin{figure}[htbp]
  \centering
  \begin{subfigure}[b]{0.45\textwidth}
    \includegraphics[width=\textwidth]{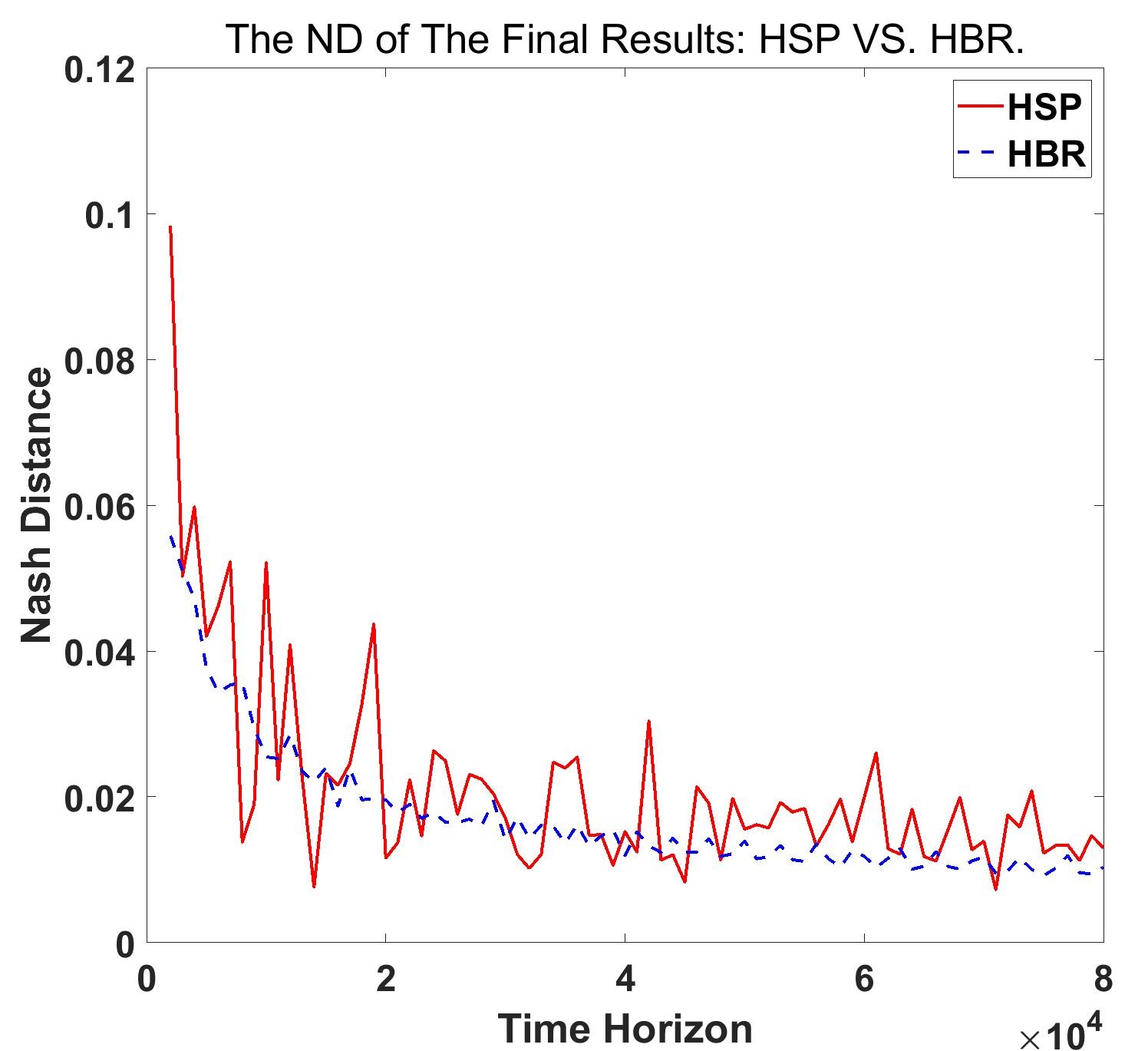}
    \caption{The $\operatorname{ND}$ of the output strategy profile for different time horizons.}
    \label{fig_ND_compare}
  \end{subfigure}
  \begin{subfigure}[b]{0.45\textwidth}
    \includegraphics[width=\textwidth]{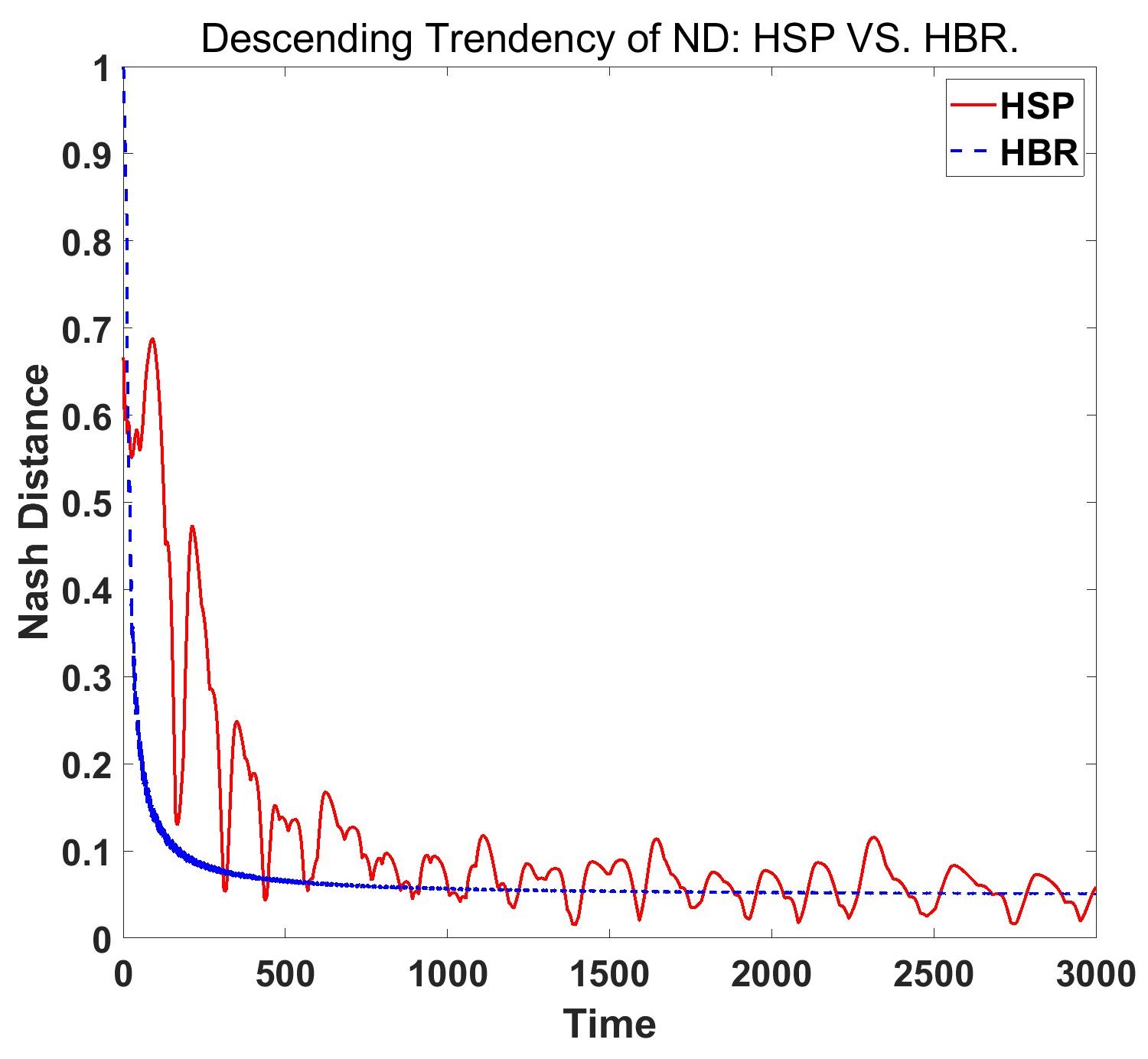}
    \caption{The descending trend for the $\operatorname{ND}$ of the time-averaged strategy profile.}
    \label{fig_ND_descending_trend}
  \end{subfigure}
  \caption{The comparison between HSP and HBR from the perspective of the convergence results and the convergence rate: game with an unique rational interior equilibrium.}
  \label{fig_HSPandHBR}
\end{figure}

\begin{figure}[htbp]
  \centering
  \begin{subfigure}[b]{0.45\textwidth}
    \includegraphics[width=\textwidth]{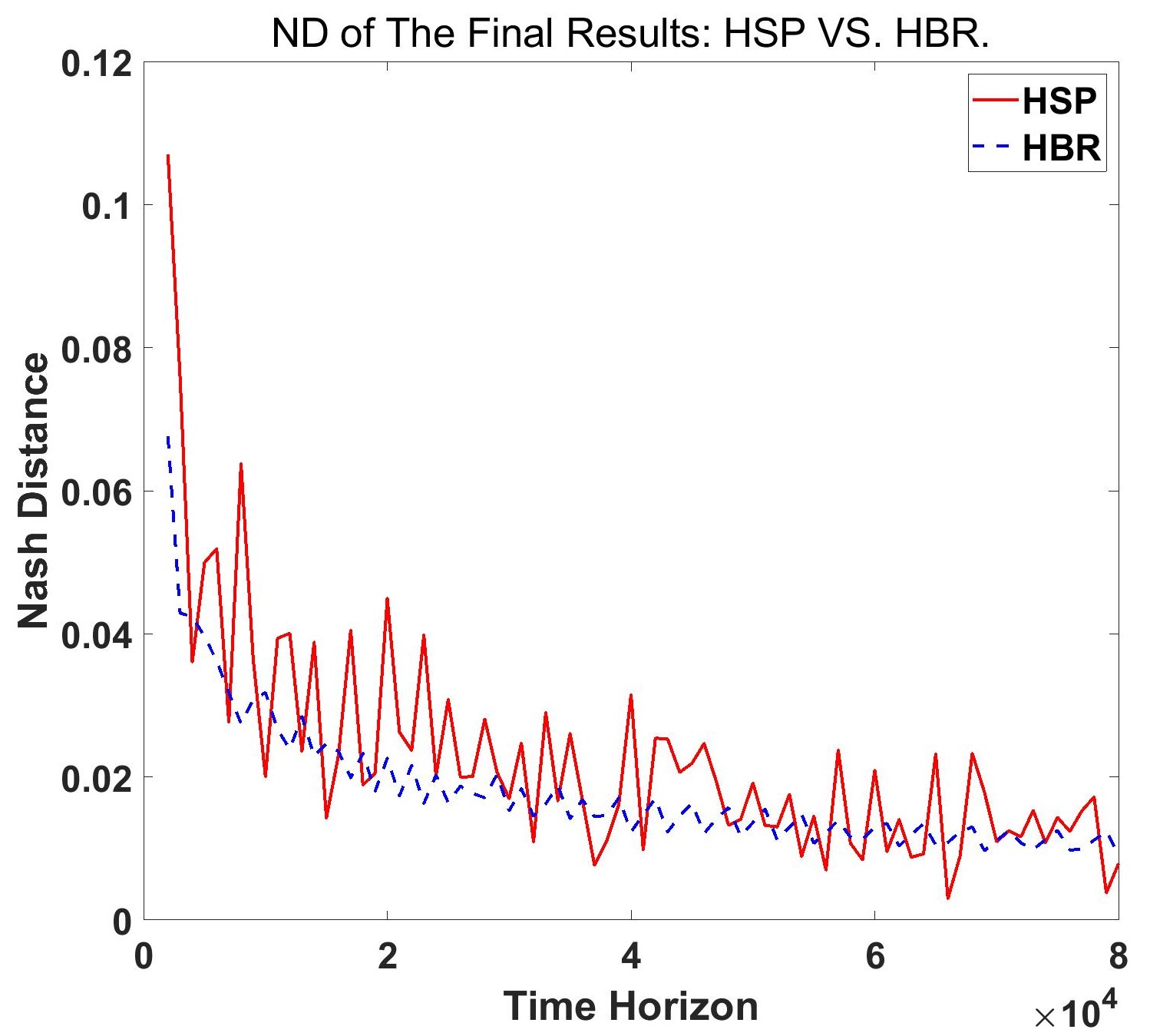}
    \caption{The $\operatorname{ND}$ of the output strategy profile for different time horizons.}
    \label{fig_ND_compare_1}
  \end{subfigure}
  \begin{subfigure}[b]{0.45\textwidth}
    \includegraphics[width=\textwidth]{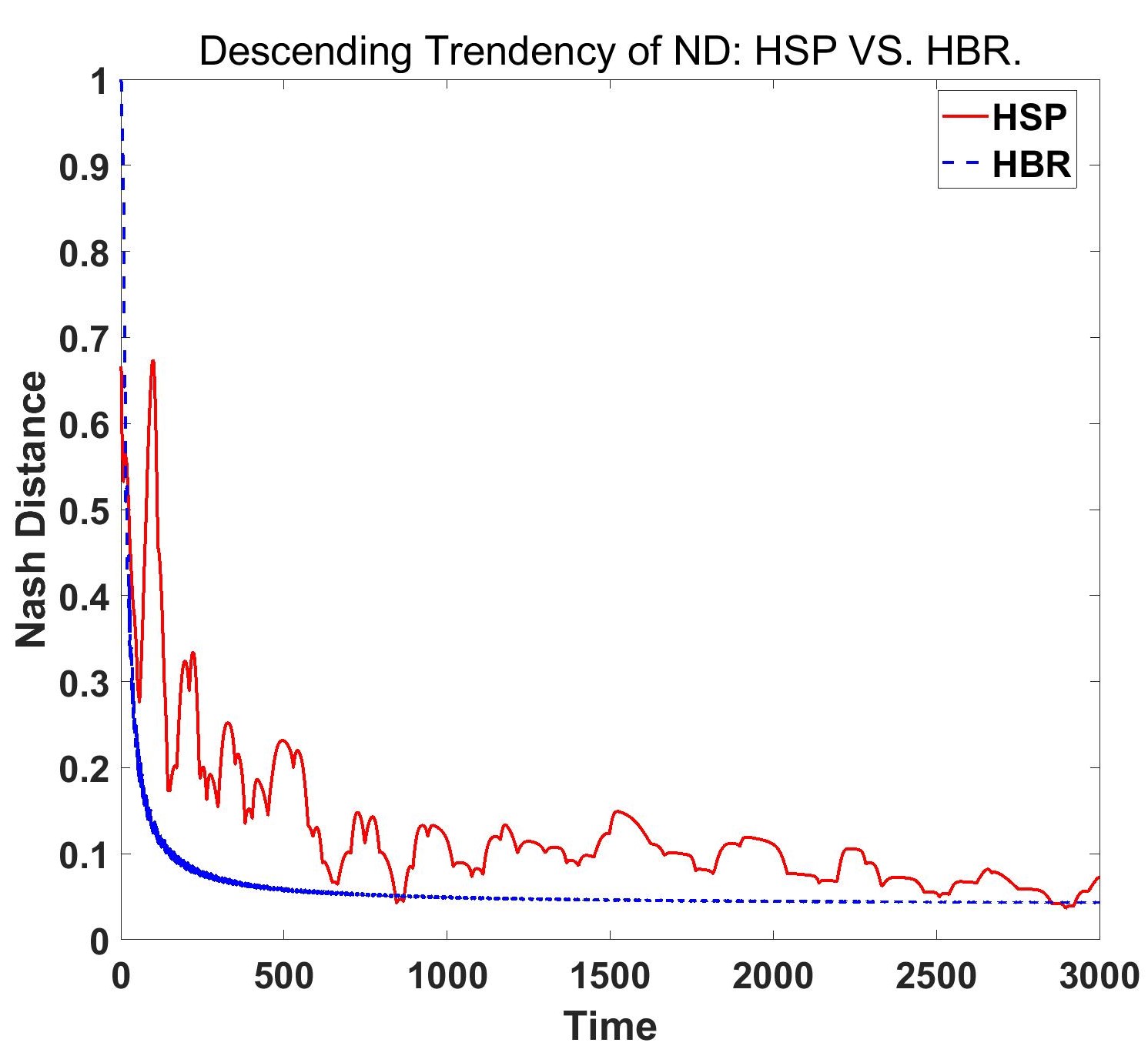}
    \caption{The descending trend for the $\operatorname{ND}$ of the time-averaged strategy profile.}
    \label{fig_ND_descending_trend_1}
  \end{subfigure}
  \caption{The comparison between HSP and HBR from the perspective of the convergence results and the convergence rate: game without rational interior equilibrium.}
  \label{fig_HSPandHBR_1}
\end{figure}

\section{Conclusion and Future Work}\label{sec_conclusion}

In this paper, we study the repeated game between one player using the Hedge algorithm and the other using the myopic best response. We prove that within this framework, when the payoff matrix is rational and the game has an interior NE, the dynamical game system enters a cycle after finite time. Moreover, within each period, the time-averaged strategy of the player using the myopic best response is an exact NE strategy. For the game with no cycle, the time-averaged strategy over the entire time horizon converges to the approximate NE. Based on these results, we propose the novel asymmetric HBR paradigm for NE-solving, which can save a substantial amount of computation costs and exhibit fast convergence rate and better stability.

In the future, we can explore the Hedge-myopic system for general games which have complex structures of NE. On the other hand, the periodicity of the Hedge-myopic system actually rules out the possibility of stage strategies converging to NE. Therefore, it is significant to 
consider how to modify the HBR paradigm so that the stage strategy can converge to NE strategy, i.e., the last-iterate property. We leave this also as future work. 

\section{Acknowledgement}

This work was supported by the National Key Research and Development Program of China under grant No.2022YFA1004600, the Natural 
Science Foundation of China under Grant T2293770, the Major Project on New Generation of Artificial Intelligence from the Ministry of Science and Technology (MOST) of China under Grant No. 2018AAA0101002, the Strategic Priority Research Program of Chinese Academy of Sciences under Grant No. XDA27000000.

\bibliographystyle{plainnat}
\bibliography{ms.bib}

\appendix

\section{Proof of Proposition \ref{Lem_interior}}\label{Adx_lem_interior}
\begin{proof}
    First, we claim that there does not exist $\textbf{x}\in \Delta(\mathcal{I})$ such that the payoffs obtained by the pure strategies of player Y are all greater than the game value $v^\ast$, i.e., for all $\textbf{x}\in \Delta(\mathcal{I})$, we have 
    \begin{equation*}
        \min\limits_{j\in \mathcal{J}} x^T A y^j \leq v^\ast.
    \end{equation*}
    If not, suppose that for $\textbf{x}_b \in \Delta(\mathcal{I})$, $\textbf{x}_b^T A \textbf{y}^j > v^\ast$ for all $j \in\mathcal{J}$. Then, for the NE strategy $\textbf{y}^\ast$ of player Y, since the equilibrium is interior, to play with $\textbf{y}^\ast$ of player Y, the loss of player X will be the same (which actually equals $v^\ast$) no matter which pure strategy is taken by player X. Then   
    we have
    \begin{align*}
        v^\ast = \textbf{x}_b^T A \textbf{y}^\ast = \sum\limits_{j=1}^n \textbf{y}_j^\ast \textbf{x}_b^T A \textbf{y}^j > \sum\limits_{j=1}^n \textbf{y}_j^\ast v^\ast = v^\ast,
    \end{align*}
    which leads to contradiction. 

    Next, we investigate the vertices of the region $Z_u$. 
    
    Note that the point corresponding to the NE strategy of player X is the unique solution of the $n$ systems of linear equations: $A_k \textbf{x} = \textbf{c},\ j = 1, 2, \ldots, n$, where the coefficient matrix $A_k$ is
    \begin{equation}
        \begin{pmatrix}
            a_{1,1} & a_{2,1} & \cdots & a_{n,1}\\
            \vdots & \ & \  & \vdots \\
            a_{1, k-1} & a_{2, k-1} & \cdots & a_{n, k-1} \\
            a_{1, k+1} & a_{2, k+1} & \cdots & a_{n, k+1} \\
            \vdots & \ & \  & \vdots \\
            a_{1, n} & a_{2, n} & \cdots & a_{n, n} \\
            1 & 1 & \cdots & 1
        \end{pmatrix},
    \end{equation}
    and $\textbf{c} = [v^\ast, v^\ast, \cdots, v^\ast, 1]^T$, whose first $n-1$ components are $v^\ast$ and the $n$-th component is $1$. 

    Now, consider the equation system $A_k \textbf{x} = \textbf{c} + \delta_c$, where $\delta_c = [\eta \delta^2/8, \eta \delta^2/8, \cdots, \eta \delta^2/8, 0]$, whose first $n-1$ components are $\eta \delta^2/8$ and the $n$-th component is $0$. By \textbf{Assumption 2}, $A_k$ is non-singular, the equation system admits an unique solution, denoted by $\textbf{x}^k$. 
    
    Firstly, we prove that $\textbf{x}^k$ belongs to the region $Z_u$ and is a vertex of $Z_u$. Since the strategy $\textbf{x}^k$ is the solution to the linear equation system $A_k \textbf{x} = c + \delta_c$, i.e., $(\textbf{x}^k)^T A \textbf{y}^j = v^\ast + \eta \delta^2/8$, thus  $(\textbf{x}^k)^T A \textbf{y}^j > v^\ast$ for $j \neq k$. By the arguments at the beginning of this proof,  we have $(\textbf{x}^k)^T A \textbf{y}^k < v^\ast \leq v^\ast + \eta \delta^2/8$. Combining these equations and inequality, we have $A^T \textbf{x}^k \leq \textbf{b}$, where $\textbf{b}$ is the vector defined in \textbf{Claim 2}. That proves  $\textbf{x}^k \in Z_u$. On the other hand, by the $n-1$ equations,  $\textbf{x}^k$ lies on the facets of the polyhedron $Z_u$, thus $\textbf{x}^k$ is a vertex of $Z_u$.

    Secondly, we prove that $\textbf{x}^k$ is an interior strategy for all $k$. Since $\textbf{x}^\ast$ is the solution of the linear equation system $A_k \textbf{x} = \textbf{c}$, it follows that $A_k (x^k - x^\ast) = \delta_c$ and thus,
    \begin{equation*}
        \| \textbf{x}^k - \textbf{x}^\ast \|_2 \leq \| A_k^{-1} \|_2 \| \delta_c \|_2 \leq \lambda_m \sqrt{n} \frac{\eta \delta^2}{8},
    \end{equation*}
    where $\lambda_m = \max\nolimits_{k=1, 2, \ldots, n}\| A_k^{-1} \|_2$. $\lambda_m$ is finite since $A_k$ is non-singular for all $k$.
    Since $\textbf{x}^\ast$ is an interior strategy and $\eta$ is small enough, we know that for any $k$, $\textbf{x}^k$ is also an interior point, i.e., $\exists\ \varepsilon^k >0 $, s.t. $x_i^k > \varepsilon^k, \ \forall\ i.$

    Lastly, let $\varepsilon_d = \min\nolimits_{k=1, 2, \ldots, n} \varepsilon^k.$ By Lemma \ref{Lem_polyhedra}, for all $\textbf{x}\in Z_u$, there exists $\lambda_1, \lambda_2, \cdots, \lambda_n$, satisfying $\lambda_k\geq 0$ and $\sum\nolimits_{k=1}^n \lambda_k = 1$, such that $\textbf{x} = \sum\nolimits_{k=1}^n \lambda_k \textbf{x}^k$. Then, for all $i = 1, 2, \ldots, n$, we have
    \begin{align*}
        x_i &= \sum\nolimits_{k=1}^n \lambda_k x_i^k > \sum\nolimits_{k=1}^n \lambda_k \varepsilon^k \geq \sum\nolimits_{k=1}^n \lambda_k \varepsilon_d = \varepsilon_d,
    \end{align*}
    which completes the proof.
\end{proof}

\section{Proof of Theorem \ref{Lem_finite_values}}\label{Appendix_C}

\begin{proof}
    First, define the set  
\begin{equation}\label{Def_Xepsilon}
    \mathcal{X}_\varepsilon = \{ \textbf{x}: \exists\ i\in \mathcal{I}\ \text{such that}\ x_i\leq \varepsilon\}.
\end{equation}
For $\textbf{x}\in \mathcal{X}_\varepsilon$, denote the index $i$ such that $x_i \leq \varepsilon$ to be $i(\textbf{x})$. The index $i(\textbf{x})$ might not be unique, but it does not affect the following discussion.

Note that for $\textbf{x} \in \mathcal{X}_\varepsilon$, $Q(\textbf{x})$ has a lower bound since 
\begin{align*}
    Q(\textbf{x}) &= -\sum\limits_{i=1}^n x_i^\ast \ln x_i \geq - x_{i(\textbf{x})}^\ast \ln x_{i(\textbf{x})}\\
    &  \geq - x_{i(\textbf{x})}^\ast \ln \varepsilon \geq  \min\limits_{i\in \mathcal{I}}\{x_i^\ast\} \ln \frac{1}{\varepsilon},
\end{align*}
where the first inequality holds because $-x_i^\ast \ln x_i>0$ for all $i\in\mathcal{I}$. 

By Lemma \ref{Lem_Qsequence_bounded}, we know that the Q-sequence $\{Q_t\}$ is bounded, i.e., $0 < Q_t \leq M_Q$ for all $t$. By the definition of $Q(\textbf{x})$, given $M_Q$, there must exist $\varepsilon_Q > 0$ such that $x_{i, t} > \varepsilon_Q$ for all $i\in \mathcal{I}$ and $t\geq 1$.

% Thus we have $\textbf{x}_t\notin \mathcal{X}_{\varepsilon_Q}$ for all $t\geq 1$. If not, there exists some time $t^\prime$ such that $\textbf{x}_{t^\prime}\in \mathcal{X}_{\varepsilon_Q}$ and then we have $  M_Q < - \min\nolimits_{i\in \mathcal{I}}\{x_i^\ast\} \ln \varepsilon_Q \leq Q(x_{t^\prime}) \leq M_Q$, which leads to contradiction. Therefore, we have $\textbf{x}_t\notin \mathcal{X}_{\varepsilon_Q}$ for all $t$, which means that $x_{i, t} > \varepsilon_Q$ for all $i\in \mathcal{I}$ and $t\geq 1$.

Recall the formula \eqref{Def_goodness} saying $ Q(\textbf{x}_t) = \ln \left( \sum\nolimits_{i=1}^n e^{\eta \sum_{\tau=1}^{t-1} (v^\ast - e_i^T A \textbf{y}_\tau)} \right).$ For $i = 1, 2, \cdots, n$, let $R_{i, t} = \sum\nolimits_{\tau=1}^{t-1} (v^\ast - e_i^T A \textbf{y}_\tau)$
and let $\mathbf{R}_t = (R_{1, t}, R_{2, t}, \cdots, R_{n, t})$. Then, the formula \eqref{Eq_hedge_formula} can be written as
\begin{equation}\label{Def_mapping_phi}
    x_{i, t} = \frac{\exp(-\eta \sum\limits_{\tau=1}^{t-1}e_i^TA\textbf{y}_{\tau})}{\sum\limits_{j=1}^n\exp(-\eta \sum\limits_{\tau=1}^{t-1}e_j^TA\textbf{y}_{\tau})} = \frac{\exp(\eta \sum\limits_{\tau=1}^{t-1}(v^\ast - e_i^TA\textbf{y}_{\tau}))}{\sum\limits_{j=1}^n\exp(\eta \sum\limits_{\tau=1}^{t-1}(v^\ast - e_j^TA\textbf{y}_{\tau}))} = \frac{\exp(\eta R_{i,t})}{\sum\limits_{j=1}^n\exp(\eta R_{j, t})},
\end{equation}
implying that for a given vector $\mathbf{R}_t$, only one strategy $\textbf{x}_t$ can be obtained. 
Denote the mapping from  $\mathbf{R}_t$ to $\textbf{x}_t$ by $\phi: \mathbb{R}^n \to \Delta(\mathcal{I})$. 
Since $\exp(\eta R_{i,t}) > 0$ for all $i$ and $t$, the range of the mapping $\phi$ is $\Delta(\mathcal{I}) \setminus \partial(\Delta(\mathcal{I}))$. 

Now consider the inverse mapping of $\phi$, i.e., $\phi^{-1}: \Delta(\mathcal{I}) \setminus \partial(\Delta(\mathcal{I})) \to \mathbf{R}_t.$ 

By the equation \eqref{Def_mapping_phi}, we have 
\begin{equation}\label{Eq_condition1}
    R_{i, t} - R_{n, t} = \frac{1}{\eta}(\ln x_{i, t} - \ln x_{n, t}).
\end{equation}
For $R_{1, t}, R_{2, t}, \cdots, R_{n, t}$, we also have 
\begin{equation}\label{Eq_condition2}
    \sum\limits_{i=1}^n x_i^\ast R_{i, t} = \sum\limits_{i=1}^n x_i^\ast \left(\sum\limits_{\tau=1}^{t-1} (v^\ast - e_i^T A \textbf{y}_\tau)\right) = \sum\limits_{\tau=1}^{t-1} \sum\limits_{i=1}^n x_i^\ast  (v^\ast - e_i^T A \textbf{y}_\tau) = \sum\limits_{\tau=1}^{t-1} (v^\ast - x^\ast A \textbf{y}_\tau) = 0
\end{equation}
since the game is zero-sum and $\textbf{x}^\ast$ is the NE strategy of player X. Combining the equation \eqref{Eq_condition1} and \eqref{Eq_condition2}, we obtain that 
\begin{equation}\label{Def_inverse_phi}
    R_{i, t} = \frac{1}{\eta} \ln x_{i, t} - \frac{1}{\eta} \sum_{i=1}^n x_i^\ast \ln x_{i, t} = \frac{1}{\eta} \ln x_{i, t} + \frac{1}{\eta} Q(\textbf{x}_t)
\end{equation}
which means that for a given vector $\textbf{x}_t$, only one regret vector $\mathbf{R}_t$ can be obtained.

Note that for all $t$, we have $\varepsilon_Q < x_{i, t} < 1$ for all $i$ and $0 < Q(\textbf{x}_t) \leq M_Q$. Hence, we have 
\begin{equation*}
    \frac{1}{\eta} \ln \varepsilon_Q \leq R_{i, t} \leq \frac{1}{\eta} M_Q
\end{equation*}
for all $i$ and $t$, which means that the possible values of $R_{i, t}$ in the Hedge-myopic system are in the range $[\frac{1}{\eta}\ln \varepsilon_Q, \frac{1}{\eta} M_Q].$ 

On the other hand, by the definition, $R_{i, t}=\sum\nolimits_{\tau=1}^{t-1} (v^\ast - e_i^T A \textbf{y}_\tau)$, in which $e_i^T A y_\tau$ refers to one of the elements of the matrix $A$, and $v^\ast$ is the game value, which is a fixed rational number. Hence, $R_{i, t}$ is also rational and  can only take values that are integer multiples of a fraction, rather than being able to take all rational numbers.
Combing with the boundness of $R_{i, t}$,  we can infer that $R_{i, t}$ can only take finite possible values. This implies that $\mathbf{R}_t$ can only take finite values and thus $x_t$ can only take finite values, which completes the proof.
\end{proof}
% Let $r^j = \vec{v}^\ast - A y^j$ where $y^j$ is a pure strategy as defined before and $\vec{v}^\ast = (v^\ast, v^\ast, \cdots, v^\ast)$ is a n-dimensional vector with each element being $v^\ast$.

\begin{figure}[htbp]
    \centering
    \includegraphics[width=\textwidth]{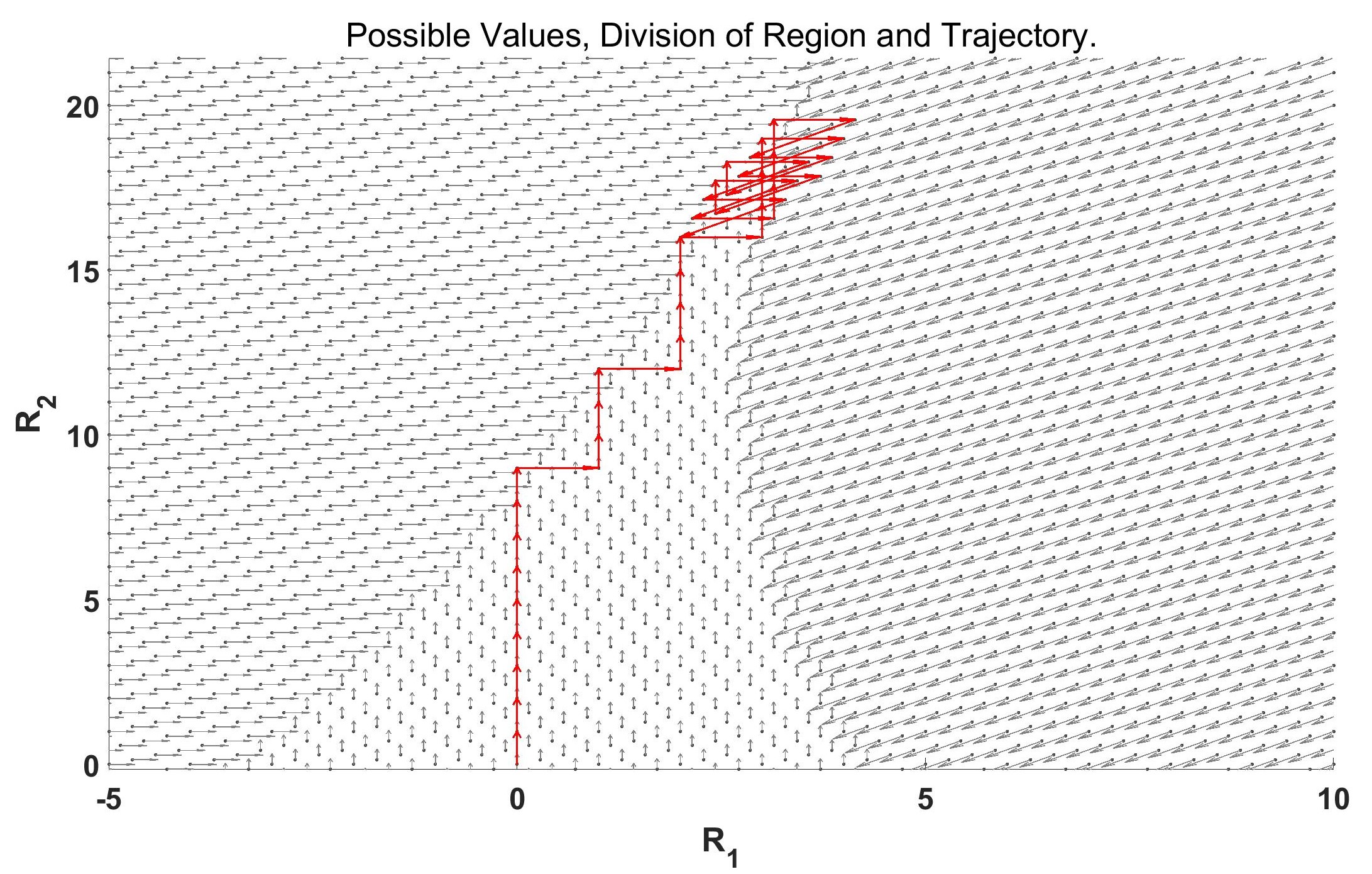}
    \caption{Possible values for the $\mathbf{R}_t$ vector for Example \ref{ex_theorem_explain}.}
    \label{fig_mesh_grid}
\end{figure}

The evolutionary direction of $\mathbf{R}_t$ in the Hedge-myopic system for a special $3\times 3$ game(Example \ref{ex_theorem_explain}) is shown in Figure \ref{fig_mesh_grid}. In this figure, each point represents a possible value of $\mathbf{R}_t$, and has an arrow pointing to another point, indicating the best response action of player Y when player X adopts the strategy corresponding to that point. The entire range of possible values for $\mathbf{R}_t$ is divided into three regions: below, left, and upper right. In each region, the update direction of $\mathbf{R}_t$ is opposite to the relative location of that region. For example, the points in the lower region point upwards. The red line represents the actual evolutionary path of the Hedge-myopic system.

\section{Proof of Lemma \ref{Lem_X_cycle_to_Y_cycle}}\label{Appendix_A}

\begin{proof}
    By the updating rule of $\textbf{y}_t$, we know that $\textbf{y}_t$ is fully determined by the value of $\textbf{x}_t$. Thus, it is natural that if the sequence $\textbf{x}_t$ is periodic, then the corresponding sequence of $\textbf{y}_t$ is also periodic.

    Suppose for some $t, T$, $x_{i, t+T} = x_{i, t}, \ \forall i = 1, 2, \cdots, n.$ By the formula \eqref{Eq_hedge_formula}, each element of $\textbf{x}_t$ is strictly positive, i.e., $x_{i, t} > 0,\ \forall i\in \mathcal{I}$. Then, for $i = 1, 2, \cdots, n-1$, we have
    \begin{align*}
        \frac{x_{i, t+T}}{x_{n, t+T}} &= \frac{x_{i, t}}{x_{n, t}} \\
        \Rightarrow \quad  \frac{\exp(-\eta \sum\limits_{\tau=1}^{t+T-1} e_i^T A\textbf{y}_\tau)}{\exp(-\eta \sum\limits_{\tau=1}^{t+T-1} e_n^T A\textbf{y}_\tau)} &= \frac{\exp(-\eta \sum\limits_{\tau=1}^{t-1} e_i^T A\textbf{y}_\tau)}{\exp(-\eta \sum\limits_{\tau=1}^{t-1} e_n^T A\textbf{y}_\tau)} \\
       \Rightarrow \quad \exp(-\eta \sum\limits_{\tau=t}^{t+T-1} e_i^T A\textbf{y}_\tau) &= \exp(-\eta \sum\limits_{\tau=t}^{t+T-1} e_n^T A\textbf{y}_\tau) \\
       \Rightarrow \quad\quad\quad \quad \sum\limits_{\tau=t}^{t+T-1} e_i^T A\textbf{y}_\tau &= \sum\limits_{\tau=t}^{t+T-1} e_n^T A\textbf{y}_\tau \\
      \Rightarrow \quad  e_i^T A \left(\frac{1}{T} \sum\limits_{\tau=t}^{t+T-1}\textbf{y}_\tau \right) &= e_n^T A \left(\frac{1}{T} \sum\limits_{\tau=t}^{t+T-1}\textbf{y}_\tau \right) .
    \end{align*}
    Since the game admits an unique rational interior equilibrium, the above equations implies that $\frac{1}{T} \sum\nolimits_{\tau=t}^{t+T-1}\textbf{y}_\tau $ is the NE strategy of player Y, i.e.     
    $\frac{1}{T} \sum\nolimits_{\tau=t}^{t+T-1}\textbf{y}_\tau = \textbf{y}^\ast$. 
\end{proof}

\end{document}